\documentclass{article}
\usepackage{fancyhdr}

\usepackage[british]{babel}
\usepackage[letterpaper,top=2cm,bottom=2cm,left=3cm,right=3cm,marginparwidth=1.75cm]{geometry}
\usepackage{amsmath, amsbsy, amssymb, amsthm}
\usepackage{graphicx}
\usepackage{xspace}
\usepackage[T1]{fontenc}
\usepackage[utf8]{inputenc}
\usepackage{float}
\usepackage{caption}
\usepackage{quiver}
\usepackage{yhmath}
\usepackage[nottoc,numbib]{tocbibind}
\usepackage{scalerel}
\usepackage{footmisc}
\usepackage[useregional]{datetime2}
\usepackage[colorlinks=true, allcolors=blue]{hyperref}
\usepackage[maxbibnames=99, style=alphabetic, backend=biber, url=false, doi=false, isbn=false, date=year, hyperref, backref, backrefstyle=none]{biblatex}
\DefineBibliographyStrings{english}{
  backrefpage={$\uparrow$\!\!},
  backrefpages={$\uparrow$\!\!}
}
\usepackage{filecontents}
\DTMlangsetup[en-GB]{showdayofmonth=false}
\usepackage{tikz}
\newcommand*\circled[1]{\tikz[baseline=(char.base)]{
    \node[shape=circle,draw,inner sep=0.5pt] (char) {#1};}}
\usepackage{stmaryrd}

\newtheorem{proposition}{Proposition}[section]
\newtheorem{theorem}[proposition]{Theorem}
\newtheorem{lemma}[proposition]{Lemma}
\newtheorem{example}[proposition]{Example}
\newtheorem{definition}[proposition]{Definition}
\newtheorem{definition-lemma}[proposition]{Definition-Lemma}
\newtheorem{remark}[proposition]{Remark}

\newtheorem{corollary}[proposition]{Corollary}

\newtheorem*{question*}{Questions}

\newtheorem{terminology}[proposition]{Terminology}

\numberwithin{equation}{proposition}

\makeatletter
\newcommand{\etale}{\'etal\@ifstar{\'e}{e\xspace}}
\makeatother
\newcommand\blfootnote[1]{%
  \begingroup
  \renewcommand\thefootnote{}\footnote{#1}%
  \addtocounter{footnote}{-1}%
  \endgroup
}
\newcommand{\Addresses}{{
  \bigskip
\noindent\textsc{School of Mathematics, Institute for Advanced Study,}\par\nopagebreak
\noindent Email: \texttt{tong.g.h.zhou@gmail.com}
  }}

\newcommand{\isoto}{\xrightarrow{\raisebox{-0.5ex}[0ex][0ex]{$\sim$}}}
\newcommand{\isoot}{\xleftarrow{\raisebox{-0.5ex}[0ex][0ex]{$\sim$}}}
\newcommand{\rightthreearrows}{\mathrel{\substack{\textstyle\rightarrow\\[-0.6ex]
                      \textstyle\rightarrow \\[-0.6ex]
                      \textstyle\rightarrow}}}

\newcommand{\ZZ}{\mathbf{Z}}

\newcommand{\NN}{\mathbf{N}}

\newcommand{\FF}{\mathbf{F}}

\newcommand{\CF}{\mathcal{F}}
\newcommand{\CG}{\mathcal{G}}

\newcommand{\Gm}{\mathbf{G}_m}
\newcommand{\X}{\times}

\newcommand{\CCC}{\mathcal{C}}
\newcommand{\CH}{\mathcal{H}}
\newcommand{\CL}{\mathcal{L}}

\newcommand{\AAA}{\mathbf{A}}
\newcommand{\DD}{\mathbb{D}}

\newcommand{\CO}{\mathcal{O}}

\newcommand{\RHom}{R\underline{Hom}}

\newcommand{\Spa}{\mathrm{Spa}}
\newcommand{\DX}{D^{(b)}_{zc}(X)}
\newcommand{\Dbzc}{D^{(b)}_{zc}}
\newcommand{\Dbzctf}{D^{(b)}_{zctf}}
\newcommand{\CX}{\mathcal{X}}

\newcommand{\cone}{\mathrm{cone}}
\newcommand{\Ker}{\mathrm{Ker}}
\newcommand{\Perv}{\mathrm{Perv}}
\newcommand{\Sh}{\mathrm{Sh}}
\newcommand{\piwa}{(1)^\tau}
\newcommand{\miwa}{(-1)^\tau}

\title{Vanishing Cycles for Zariski-Constructible Sheaves on Rigid Analytic Varieties}
\author{Tong Zhou}
\date{}
\begin{document}
\maketitle
\begin{abstract}
\blfootnote{March 2025}
We develop a theory of nearby and vanishing cycles in the context of finite-coefficient Zariski-constructible sheaves over a non-archimedean field which is non-trivially valued, complete, algebraically closed, and of mixed characteristic or equal characteristic zero. Apart from basic properties, we show that they preserve Zariski-constructibility, have a Milnor fibre interpretation, satisfy Beilinson's gluing construction, are perverse t-exact, and commute with Verdier duality.

\end{abstract}
\renewcommand{\baselinestretch}{1.0}\normalsize
\tableofcontents
\renewcommand{\baselinestretch}{1.0}\normalsize
\section{Introduction}\label{sec_intro}
Nearby and vanishing cycles were introduced by Alexandre Grothendieck in \cite{SGA7}, to study changes of the sheaf cohomology in a $1$-parameter family. They are as fundamental as the six functors, and have been extensively studied in almost all contexts of (or similar to) constructible sheaves (with the characteristic of the coefficient prime to that of the base field): \etale cohomology of schemes (\cite[Exposé I, XIII]{SGA7}), Betti cohomology of complex analytic varieties (\cite[Exposé XIV]{SGA7}), D-modules and mixed Hodge modules (\cites{kashiwara_vanishing_1983, saito_modules_1988}), and \etale cohomology of formal schemes (\cite{berkovich_vanishing_1994}). There is also a different construction of the unipotent parts of nearby and vanishing cycles of perverse sheaves by Alexander Beilinson (\cite{beilinson_how_1987}), applicable in various contexts.\\

In non-archimedean geometry, there exists so far two definitions of nearby and vanishing cycles, by Vladimir Berkovich (\cite{berkovich_vanishing_1996}) and Lorenzo Ramero (\cite{ramero_hasse-arf_2012}) respectively. However, a full theory is missing, partly due to the lack of a good notion of constructibility for \etale sheaves on non-archimedean spaces. Recently, a six-functor formalism and a theory of perverse sheaves have been constructed for Zariski-constructible sheaves on rigid analytic varieties, by Bhargav Bhatt and David Hansen (\cite{bhatt_six_2022}). These are direct analogues of constructible sheaves in complex analytic geometry and behave similarly. In this work, we develop a full theory of nearby and vanishing cycles in this context. It will be applied in \cite{zhou_microlocal_2025} to develop a microlocal sheaf theory in the sense of \cite{kashiwara_sheaves_1990}.\\

Let $K$ be a non-archimedean field which is non-trivially valued, complete, algebraically closed, and of characteristic $(0,p)$, $p\geq0$. Let $\Lambda$ be either $\FF_{\ell^r}$ or $\ZZ/\ell^r$, with $\ell=p\neq0$ allowed. We work with rigid analytic varieties over $K$ (\textit{i.e.}, adic spaces locally topologically of finite type over $\Spa(K,K^\circ)$) and $\Lambda$-valued Zariski-constructible \etale sheaves\footnote{This will henceforth be referred to as “our context”.}. In §\ref{sec_defofvan}, we define the nearby and vanishing cycles, construct the distinguished triangle $i^*j_*\xrightarrow{\mathrm{sp}}\psi\xrightarrow{\iota}\psi\miwa\rightarrow$ (Lemma \ref{lem_T-1}) and \textit{can} and \textit{var} triangles (Definition \ref{def_sp_can_var}), and prove basic properties of them. In particular, nearby and vanishing cycles commute with smooth pullbacks ($\ell\neq p$) and quasi-compact quasi-separated pushforwards (Lemma \ref{lem_smoothqcqs}), and with analytification (Lemma \ref{lem_comparisonphi}). The “$\miwa$” above denotes the Iwasawa twist, which will be briefly reviewed. It is a formalism introduced by Beilinson and developed in \cite[\nopp §2]{lu_duality_2019} for eliminating the choice of a topological generator of the local fundamental group, improving the “canonicity” of constructions. In §\ref{sec_finiteness_Milnor}, we prove nearby and vanishing cycles preserve Zariski-constructibility (Proposition \ref{prop_phipreserveszc}), and give a Milnor fibre (tube) interpretation of their stalks ($\ell\neq p$) (Proposition \ref{prop_nearbyfibinterp}), analogous to the classical version in the complex analytic context. We also prove some preliminary results about perverse sheaves, in particular, a General Artin-Grothendieck Vanishing ($\ell\neq p$) (Proposition \ref{prop_artinvanishing}), analogous to its counterpart in the complex analytic context (\cite[\nopp 10.3.17]{kashiwara_sheaves_1990}). In §\ref{sec_beilinson_gluing} ($\ell\neq p$), we discuss Beilinson's construction of the unipotent nearby and vanishing cycles and the maximal extension functor, and his gluing theorem, in our context. In §\ref{sec_perversity_duality} ($\ell\neq p$), we prove nearby and vanishing cycles are perverse t-exact and commute with Verdier duality, which we summarise as the following theorem.

\begin{theorem}[Theorem \ref{thm_psi_t_exact_duality}, Corollary \ref{cor_ptexactphigeneralcoefficient}, Theorem \ref{thm_phi_t_exact_duality}]
The set-up is as above, $\Lambda=\FF_{\ell^r}$ or $\ZZ/\ell^r$, $\ell\neq p$. Let $f: X\rightarrow\AAA^1$ be a map of rigid analytic varieties. Denote $\psi_f[-1]$ (resp. $\phi_f[-1]$) by $\Psi_f$ (resp. $\Phi_f$). Then:\\
    (1) $\Psi_f$ is perverse t-exact in the following sense: 
    for every $\CF\in$ $^p\!D^{\leq 0}_{zctf}(X^\times)$ and $\CG\in$ $^p\!D^{\geq 0}_{zctf}(X^\times)$ which extend to objects of $D^{(b)}_{zc}(X)$, we have $\Psi_f(\CF)\in$ $^p\!D^{\leq 0}_{zctf}(X_0)$ and $\Psi_f(\CG)\in$ $^p\!D^{\geq 0}_{zctf}(X_0)$;\\
    (2) $\Phi_f: D^{(b)}_{zctf}(X)\rightarrow D^{(b)}_{zctf}(X_0)$ is perverse t-exact;\\
    (3) for $\CF\in$ $D^{(b)}_{zc}(X^\X)$ which extends to an object of $\DX$, there is a canonical isomorphism $g: \Psi_f\mathbb{D}\CF\isoto(\mathbb{D}\Psi_f\CF)(1)$ in $D(X_0\bar\X B\mu)$;\\
    (4) for $\CF\in$ $D^{(b)}_{zc}(X)$, there is a natural isomorphism\footnote{We leave it to the reader to decide if it is canonical or not.} $h: \Phi_f\mathbb{D}\CF\isoto(\mathbb{D}\Phi_f\CF)\miwa(1)$ in $D(X_0\bar\X B\mu)$.\\\\
    Furthermore, the $\mathrm{can}$ and $\mathrm{var}$ triangles are dual to each other in the sense that we have a commutative diagram:
\[\begin{tikzcd}
	{(i^*\mathbb{D}\mathcal{F})[-1]} && {\Psi\mathbb{D}\mathcal{F}} && {\Phi\mathbb{D}\mathcal{F}} \\
	{(\mathbb{D}i^!\mathcal{F})[-1]} && {(\mathbb{D}\Psi\mathcal{F})(1)} && {(\mathbb{D}\Phi\mathcal{F})(-1)^\tau(1)}
	\arrow["\mathrm{sp}", from=1-1, to=1-3]
	\arrow["\alpha", from=1-1, to=2-1]
	\arrow["\mathrm{can}", from=1-3, to=1-5]
	\arrow["g", from=1-3, to=2-3]
	\arrow["\simeq"', from=1-3, to=2-3]
	\arrow["h", from=1-5, to=2-5]
	\arrow["{\mathbb{D}(\mathrm{cosp}(1)^\tau(-1))}", from=2-1, to=2-3]
	\arrow["{\mathbb{D}(\mathrm{var}(1)^\tau(-1))}", from=2-3, to=2-5]
\end{tikzcd}\]
\end{theorem}

Here $X_0\bar\X B\mu$ is the topos of sheaves on $X_0$ equipped with a $\mu$-action, where $\mu=\hat{\ZZ}(1)$. We refer to the main text for the constructions of the maps $g$, $h$ and $\alpha$.

\begin{remark}
    In the context of schemes, nearby and vanishing cycles are also defined and studied over higher-dimensional bases (see \cite{lu_duality_2019} and references therein). No such formalism exists in analytic contexts yet. Our point of view is that, in contexts without wild ramifications (\textit{e.g.}, complex analytic, and our context), the functor $\mu hom$ as in \cite[\nopp IV.4]{kashiwara_sheaves_1990} is a linearised version of nearby and vanishing cycles “from general sources” and is a sufficient substitute, in the sense that linearisation captures enough information where there is no wild ramifications. This will be elaborated in \cite{zhou_microlocal_2025}.
\end{remark}

\begin{question*}
    (1) For a map $f: X\rightarrow Y$ of rigid analytic varieties and $\CF\in \DX$, we say $(f,\CF)$ is \underline{$\phi$-universally locally acyclic} ($\phi$-ULA) if for every map $D=\Spa(K\langle T\rangle)\rightarrow Y$ from the closed unit disc, one has $\phi_{f_D}(\CF_D)=0$, where $f_D: X_D\rightarrow D$ and $\CF_D\in\Dbzc(X_D)$ are the base changes. How does $\phi$-ULA compare with ULA as defined in \cite[\nopp IV.2.1]{fargues_geometrization_2021}?\\
    (2) Is there a Künneth formula and Thom-Sebastiani theorem (\cite{illusie_around_2017}) in this context?
\end{question*}
 
\section*{Conventions}
We use the six-functor formalism of \etale sheaves on adic spaces constructed in \cite{huber_etale_1996}, enhanced to $\infty$-categories in \cite[\nopp §9]{zavyalov_some_2024}. The derived categories $D(X,\Lambda)$ are in the $\infty$-categorical sense. $\Sh(X,\Lambda)$ denotes the abelian category of sheaves. The coefficient $\Lambda$ is often omitted. All sheaf-theoretic functors are derived. By a “local system” $\CL$ on $X$ we mean an object of $\DX$ such that each $\CH^i(\CL)$ is locally constant with finite type stalks.\\\\
By a “rigid analytic variety” we mean an adic space locally topologically of finite type over $\Spa(K,K^\circ)$. We use the terms “analytic closed subset” and “Zariski-closed subset” interchangeably. For a Huber ring $A$ topologically of finite type over $K$, we abbreviate $\Spa(A,A^\circ)$ as $\Spa(A)$. For $r\in|K^\X|$, $D(r):=\Spa(K\langle \frac{T}{r}\rangle)$ denotes the closed unit disc of radius $r$, $D(1)$ is abbreviated as $D$, punctured discs $D(r)-\{0\}$ are denoted by $D^\X(r)$.\\\\
We use “$\simeq$” to denote a canonical isomorphism, with the morphism clear from the context, and “$\cong$” to denote an isomorphism without naturality or canonicity assertions. We sometimes use “the” instead of “a” when refering to an object determined up to a contractible space of choices (\textit{e.g.} inverses and compositions).
\section*{Acknowledgements}
I am very grateful to Bhargav Bhatt for numerous valuable discussions. I also want to express my sincere thanks to Bogdan Zavyalov for many insightful suggestions, to Hiroki Kato for clarifying discussions on materials related to §\ref{sec_beilinson_gluing}, and to Piotr Achinger, Grigory Andreychev, Yash Deshmukh, Bradley Dirks, David Hansen and Luc Illusie for their interest, conversations and correspondences. This work is done during my stay at the Institute for Advanced Study. I want to thank the Institute for providing a beautiful environment.
\section{Definitions and basic properties}\label{sec_defofvan}

In this section, we define and study basic properties of nearby and vanishing cycles on rigid analytic varieties. The set-up is as in §\ref{sec_intro}, $\Lambda=\FF_{\ell^r}$ or $\ZZ/\ell^r$, with $\ell=p\neq0$ allowed.

\begin{definition}[nearby and vanishing cycles]\label{def_van}
    Let $f: X\rightarrow \AAA^1$ be a map of rigid analytic varieties, $X_0:=X\times_{\AAA^1} 0$, $X^\times:=X\times_{\AAA^1}\Gm$, and $\CF\in D^{}(X^\times)$. The \underline{nearby cycle} $\psi_f(\CF)$ of $\CF$ with respect to $f$ is an object of $D^{}(X_0)$ carrying a continuous action of $\pi^{\mathrm{fin}}_1(\Gm,1)=\hat{\ZZ}(1)=\mu:=\varprojlim\mu_n$ (the \underline{monodromy action})\footnote{Here, $\pi^{\mathrm{fin}}_1$ denotes the fundamental group classifying finite \etale surjective maps, and $\mu_n$ denotes the $n$-th roots of unity in $K$. In the following, we use $\pi^{\mathrm{fin}}_1(\Gm,1)$, $\hat{\ZZ}(1)$ and $\mu$ interchangeably.} defined as follows: consider the diagrams, where the second is the base change to $X$ of the first:
\[\begin{tikzcd}
	&& {\mathbf{G}_m} &&&& {X^{\times}_n} \\
	0 & {\mathbf{A}^1} & {\mathbf{G}_m} && {X_0} & X & {X^{\times}}
	\arrow[from=1-3, to=2-2]
	\arrow["{z\mapsto z^n}", from=1-3, to=2-3]
	\arrow["{e_n}"', from=1-3, to=2-3]
	\arrow["{j_n}"', from=1-7, to=2-6]
	\arrow["{p_n}", from=1-7, to=2-7]
	\arrow[hook, from=2-1, to=2-2]
	\arrow[hook', from=2-3, to=2-2]
	\arrow["i"', hook, from=2-5, to=2-6]
	\arrow["j", hook', from=2-7, to=2-6]
\end{tikzcd}\]
    Then $\psi_f(\CF):=\varinjlim_{n\rightarrow\infty} i^*j_*p_{n*}p_n^*\CF$. The $\pi^{\mathrm{fin}}_1(\Gm,1)$-action is induced by the natural actions of $\mathrm{Aut}(e_n)=\mu_n$ on $p_{n*}p_n^*\CF$.\\
    
    For $\CG\in D^{}(X)$, the \underline{vanishing cycle} $\phi_f(\CG)$ is an object of $D^{}(X_0)$ carrying a continuous action of $\pi^{\mathrm{fin}}_1(\Gm,1)$ defined as $\cone(i^*\CG\xrightarrow{\mathrm{sp}}\psi_f(\CG))$. Here $\mathrm{sp}: i^*\CG\rightarrow\psi_f(\CG)$ is induced by the adjunctions $\CG\rightarrow j_{n*}j_n^*\CG$ and called the \underline{specialisation map}. The $\pi^{\mathrm{fin}}_1(\Gm,1)$-action is induced from that on $\psi_f(\CG)$. The canonical map $\psi_f(\CG)\rightarrow\phi_f(\CG)$ is called the \underline{canonical map} and denoted by $\mathrm{can}$, the distinguished triangle $i^*\CG\xrightarrow{\mathrm{sp}}\psi_f(\CG)\xrightarrow{\mathrm{can}}\phi_f(\CG)\rightarrow$ is called the \underline{$\mathrm{can}$ triangle}.
\end{definition}

\begin{remark}
    (1) We warn the reader that, although $\varinjlim i^*j_*p_{n*}p_n^*\simeq i^*j_*\varinjlim p_{n*}p_n^*$ in the algebraic setting, this is not true in our setting, due to the non-quasi-compactness of $X^\X$ for an affinoid $X$. We use the former one because, for example, Lemma \ref{lem_T-1} does not hold for the latter in this generality (consider the example in Remark \ref{rmk_nonZariskiConst_example}).\\
    (2) In the following we often write $\psi_f$ for $f: X\rightarrow D(r)$ a map to a disc of radius $r$. By this we mean $\psi_{uf}$ where $u: D(r)\hookrightarrow\AAA^1$ is the standard open immersion. We will also often view $\psi_f$ as a functor from $D(X)$ to $D(X_0)$ via $\psi_f j^*$.
\end{remark}

\begin{remark}
    (1) Our definition is compatible with those of Berkovich (\cite[\nopp §4]{berkovich_vanishing_1996}) and Ramero (\cite[§2.2]{ramero_hasse-arf_2012}) in the following sense: in \cite[\nopp §4]{berkovich_vanishing_1996}, take $\mathbf{S}$ to be the germ $(\AAA^1,0)$, and $\mathbf{X}\rightarrow\mathbf{S}$ and $F$ to be induced from $f: X\rightarrow\AAA^1$ and $\CF$, then (Berkovich's) $\Psi_\eta(F)\simeq$ (our) $\psi_f(\CF)$; in \cite[§2.2]{ramero_hasse-arf_2012}, if one replaces $\pi_1^{\mathrm{loc.alg}}$ by $\pi^{\mathrm{fin}}_1$, then one gets our definition. \\
    (2) In \cite[§2.2]{ramero_hasse-arf_2012}, $\pi_1^{\mathrm{loc.alg}}$ is used because there the coefficient for sheaves is a ring which equals the union of its finite subrings. In this article we work only with finite coefficients. This significantly simplifies the sheaf theory, essentially because, with finite coefficients, it follows from a result of Ofer Gabber and Werner Lütkebohmert (\cite{lutkebohmert_riemanns_1993}) that sheaves do not have “wild ramifications”\footnote{By which we mean the complexity in genuinely infinite \etale coverings. This is analogous to the complexity in essential singularities of complex analytic functions, and the complexity in $p$-power-degree finite \etale coverings over a field of characteristic $p>0$.}. It would be desirable to study to what extent the following theory extends to infinite coefficients.
\end{remark}

More formally, “an object of $D^{}(X_0)$ carrying a continuous action of $\pi_1^{\mathrm{fin}}$” means the following. Let $B\mu$ be the classifying topos of $\mu$ and $S$ be the \etale topos in \emph{schemes} of $\mathrm{Spec}(K[[x]])$. Consider the product topos $X_0\bar\X S$.\footnote{We follow \cite{lu_duality_2019} and denote the fibre product of topoi by $\bar\X$ to distinguish from that of rigid spaces.} An object in $\Sh(X_0\bar\X S)$ is a triple $(A\xrightarrow{\alpha}B)$, where $A$ is a sheaf on $X_0$ with the trivial action of $\mu$, $B$ is a sheaf on $X_0$ with a continuous action of $\mu$, and $\alpha$ is an equivariant map. $X_0\bar\X S$ has an open subtopos $X_0\bar\X B\mu$ with complementary closed subtopos $X_0$.\\ 

Definition \ref{def_van} in fact defines $\psi$ (resp. $\phi$) as a functor from $D(X^\X)$ (resp. $D(X)$) to $D(X_0\X B\mu)$. In the following, unless otherwise specified, by a morphism (in particular, isomorphism) of nearby and vanishing cycle sheaves we mean a morphism of objects in $D(X_0)$ commuting with the monodromy actions, \textit{i.e.}, as objects of $D(X_0\bar\X B\mu)$.\\

With $X_0\bar\X S$ defined as above, the discussion in \cite[\nopp §2.1$\sim$2.3]{lu_duality_2019} applies in our setting. We will use the materials in §2.1$\sim$2.2 therein.

\begin{definition}[{\cite[\nopp §2.1]{lu_duality_2019}}, Iwasawa twist]\label{def_iwasawa}
    The \underline{Iwasawa algebra} of $\mu$ is the group algebra of $\mu$, defined as $R=\Lambda[[\mu]]:=\varprojlim_n\Lambda[\mu_n]$. Let $t$ be a topological generator of $\mu$. Then $tR$ as an ideal of $R$ is independent of the choice of $t$. It will be denoted by $R(1)^\tau$. For a discrete $\Lambda$-$\mu$-module $M$, its \underline{Iwasawa twists} are $R$-modules defined as $M\piwa=R\piwa\otimes_RM$ and $M\miwa=Hom_R(R\piwa,M)$. The Iwasawa twist is invertible in the sense that $R\piwa\otimes_RR\miwa\simeq R$. We denote by $\iota$ the canonical map $M\rightarrow M\miwa: x\mapsto (\sigma\mapsto\sigma x)$.
\end{definition}

A fixed $t$ induces an isomorphism $M\miwa\isoto M: (\sigma\mapsto f(\sigma))\mapsto f(t-1)$. Note that the composition of this isomorphism with $\iota$ is just $M\rightarrow (M\miwa\isoto) M: x\mapsto (t-1)x$. So one can think of $M\rightarrow M\miwa$ as a kind of monodromy without a choice of $t$.

\begin{lemma}[{\cite[\nopp §2.1]{lu_duality_2019}}]\label{lem_iwasawa_basics}
    (1) The map $M\miwa\rightarrow Z^1_{cont}(\mu,M): (\sigma\mapsto f(\sigma))\mapsto (h\mapsto f(h-1))$ is an isomorphism of $R$-modules, where $Z^1_{cont}(\mu,M)$ denotes continuous crossed homomorphisms.\\
    (2) If $\mu$ acts trivially on $M$, then the map $M(1)\simeq \hat{\ZZ}(1)\otimes_{\hat{\ZZ}}M\rightarrow M\piwa: h\otimes x\mapsto (h-1)\otimes x$ is an isomorphism; if $t-1$ acts invertibly on $M$, then the map $\iota: M\rightarrow M\miwa$ is an isomorphism.\\
    (3) Consider the exact sequence of abelian groups $0\rightarrow P\rightarrow\hat\ZZ(1)\rightarrow\ZZ_\ell(1)\rightarrow0$. We get a decomposition $M\simeq M^P\oplus\!^P\!M$. Here $M^P$ denotes the $P$-invariants of $M$, which is also the image of the projection $\kappa: M\rightarrow M$, $x\mapsto \int_{\sigma\in P}\sigma x d\sigma$ where $d\sigma$ is the Haar measure on $P$; $^P\!M$ denotes the kernel of $\kappa$. We have $M\piwa\simeq M^P\piwa\oplus\!^P\!M$. 
\end{lemma}

\begin{remark}\label{rmk_t_non_t_decomposition}
    The Iwasawa twist and the decomposition $M\simeq M^P\oplus\!^P\!M$ are exact, thus easily induce functors on the derived category, and on the topos $X_0\bar\X B\mu$ (\textit{c.f.} §2.1 in $\textit{loc. cit.}$). Lemma \ref{lem_iwasawa_basics} implies analogous statements on $D(X_0\bar\X B\mu)$. In particular, we have canonical decompositions $\psi\simeq \psi^P\oplus\!^P\!\psi$ and $\phi\simeq \phi^P\oplus\!^P\!\phi$.
\end{remark}

\begin{lemma}\label{lem_T-1}
    Let $f: X\rightarrow\AAA^1$ be a map of rigid analytic varieties, and $\CF\in D^{}(X)$. Then, there is a canonical distinguished triangle: $i^*j_*j^*\CF\xrightarrow{\mathrm{sp}}\psi_f(\CF)\xrightarrow{\iota}\psi_f(\CF)\miwa\rightarrow$.
\end{lemma}

\begin{proof}
    Denote by $R\Gamma(\mu,-)$ the direct image map: $D^{}(X_0\bar\X B\mu)\rightarrow D^{}(X_0)$. By \cite[Lemma 2.9]{lu_duality_2019}, we have a distinguished triangle $R\Gamma(\mu,\psi(\CF))\rightarrow \psi(\CF)\xrightarrow{\iota}\psi(\CF)\miwa$, where we have omitted the $p^*$ in \textit{loc. cit.} in our notation, and the first map is the adjunction $p^*R\Gamma(\mu,-)\rightarrow id$.\\\\
    For each $n$, we have $j^*\CF\isoto R\Gamma(\mu_n, p_{n*}p_n^*j^*\CF)$ where the map is induced by the adjunction $\mathrm{id}\rightarrow p_{n*}p_n^*$, and $R\Gamma(\mu_n,(-))$ is the direct image map $D^{}(X^\X\bar\X B\mu_n)\rightarrow D^{}(X^\X)$. Consider the following diagram 
\[\begin{tikzcd}
	{X_0\bar\times B\mu_n} & {X\bar\times B\mu_n} & {X^\times\bar\times B\mu_n} \\
	{X_0} & X & {X^\times}
	\arrow[hook, from=1-1, to=1-2]
	\arrow[from=1-1, to=2-1]
	\arrow["\lrcorner"{anchor=center, pos=0.125}, draw=none, from=1-1, to=2-2]
	\arrow[from=1-2, to=2-2]
	\arrow[hook', from=1-3, to=1-2]
	\arrow["\lrcorner"{anchor=center, pos=0.125, rotate=-90}, draw=none, from=1-3, to=2-2]
	\arrow[from=1-3, to=2-3]
	\arrow["i"', hook, from=2-1, to=2-2]
	\arrow["j", hook', from=2-3, to=2-2]
\end{tikzcd}\]
We have $j_*j^*\CF\isoto R\Gamma(\mu_n, j_*p_{n*}p_n^*j^*\CF)$. Apply coherent base change (locally on $X$) (\cite[\nopp 1.17, 1.18, 2.11]{lu_duality_2019}) to the left square, we get $i^*j_*j^*\CF\isoto R\Gamma(\mu_n, i^*j_*p_{n*}p_n^*j^*\CF)$. Pass to the limit (\cite[\nopp VI.8.7.3]{SGA4}, note the hypotheses are satisfied by the finiteness of $\mu_n$), we get $i^*j_*j^*\CF\isoto \varinjlim R\Gamma(\mu_n, i^*j_*p_{n*}p_n^*j^*\CF)\simeq R\Gamma(\mu, \varinjlim i^*j_*p_{n*}p_n^*j^*\CF)=R\Gamma(\mu, \psi_f(\CF))$. Combined this with the first paragraph, we get the desired distinguished triangle.
\end{proof}

\begin{definition}\label{def_sp_can_var}
    Let $f: X\rightarrow\AAA^1$ be a map of rigid analytic varieties, and $\CF\in D^{}(X)$. Consider the following diagram:
\[\begin{tikzcd}
	{i^*\mathcal{F}} & {i^*j_*j^*\mathcal{F}} & {i^!\mathcal{F}[1]} \\
	{i^*\mathcal{F}} & {\psi_f(\mathcal{F})} & {\phi_f(\mathcal{F})} \\
	0 & {\psi_f(\mathcal{F})(-1)^\tau} & {\psi_f(\mathcal{F})(-1)^\tau}
	\arrow[from=1-1, to=1-2]
	\arrow["{=}", from=1-1, to=2-1]
	\arrow[from=1-2, to=1-3]
	\arrow["\mathrm{sp}", from=1-2, to=2-2]
	\arrow[from=1-3, to=2-3]
	\arrow["\mathrm{sp}", from=2-1, to=2-2]
	\arrow[from=2-1, to=3-1]
	\arrow["\mathrm{can}", from=2-2, to=2-3]
	\arrow["\iota", from=2-2, to=3-2]
	\arrow["\mathrm{var}", from=2-3, to=3-3]
	\arrow[from=3-1, to=3-2]
	\arrow["{=}", from=3-2, to=3-3]
\end{tikzcd}\]
Here, each row and column is a distinguished triangle, the first row is $i^*$ applied to the recollement $i_*i^!\CF\rightarrow\CF\rightarrow j_*j^*\CF\rightarrow$ (shifted), and the second column is from Lemma \ref{lem_T-1}. More precisely, we start from the top left commutative square, where the homotopy is from the composition of adjunctions $i^*\rightarrow i^*j_*j^*\rightarrow \psi_f(=\varinjlim i^*j_*p_{n*}p^*_nj^*)$. The whole digram is then determined by taking cones.\\

We will call the third column $\phi_f(\CF)\xrightarrow{\mathrm{var}}\psi_f(\CF)\miwa\xrightarrow{\mathrm{cosp}} i^!\CF[2]\rightarrow$ the \underline{$\mathrm{var}$ triangle}, the first map the \underline{variation map}, and the second map the \underline{cospecialisation} map.
\end{definition}

\begin{remark}
    Fix a topological generator $t$ of $\mu=\hat\ZZ(1)$, we get an isomorphism $\psi\miwa\simeq\psi$ (discussion after Definition \ref{def_iwasawa}). The $\mathrm{var}$ triangle and the distinguished triangle in Lemma \ref{lem_T-1} then take the forms $\phi_f(\CF)\xrightarrow{\mathrm{var}(t)}\psi_f(\CF)\xrightarrow{\mathrm{cosp}}i^!\CF[2]\rightarrow$ and $i^*j_*j^*\CF\xrightarrow{\mathrm{sp}} \psi_f(\CF)\xrightarrow{t-1}\psi_f(\CF)\rightarrow$. 
\end{remark}

We now prove some basic properties of nearby and vanishing cycles.

\begin{lemma}\label{lem_psietalebasechange}
    Let $f: X\rightarrow\AAA^1$ be a map of rigid analytic varieties, and $\CF\in D^{}(X)$. Then, for every $\lambda\in K^\times$, we have a natural isomorphism $\psi_f(\CF)\isoto\psi_{\lambda f}(\CF)$. This isomorphism is canonical up to a choice of $\bar{\lambda}=\{\lambda_n\}_{n\in\ZZ_{\geq1}}\in\varprojlim\mu_n(\lambda)$ with $\lambda_1=\lambda$, where $\mu_n(\lambda)$ denotes the $n$-th roots of $\lambda$ in $K$. Similarly for $\phi$.
\end{lemma}

\begin{proof}
    Fix $\bar{\lambda}$. Consider $\theta_{\lambda^{-1}}: \Gm\rightarrow\Gm, z\mapsto\lambda^{-1} z$. The $n$-th Kummer covering on the target base-changes to a Kummer covering $U_n\rightarrow \Gm$ of degree $n$. It is isomorphic to the standard covering $\Gm(n)=\Gm\rightarrow \Gm, z\mapsto z^n$ via the map $\alpha_n: U_n\isoto \Gm(n): (z,\mu) \text{ (with $\mu^n=\frac{z}{\lambda}$)}\mapsto\mu\lambda_n$. As $\{\lambda_n\}$ is a projective system, $\{\alpha_n\}$ also. The base-change isomorphisms on each level $n$ then induce $\psi_f(\CF)\isoto\psi_{\lambda f}(\CF)$. It is compatible with the monodromies because, by construction, it is compatible when $\pi^{\mathrm{fin}}_1(\Gm, 1)$ acts on $\psi_{\lambda f}(\CF)$ after composing with the isomorphisms $\pi^{\mathrm{fin}}_1(\Gm, 1)\isoto \pi^{\mathrm{fin}}_1(\Gm, \lambda^{-1})\isoto\pi^{\mathrm{fin}}_1(\Gm, 1)$, where the first map is induced by $\bar{\lambda}$ and the second by $\theta_{\lambda^{-1}}$, but this composition is in fact identity as $\pi^{\mathrm{fin}}_1(\Gm, 1)$ is abelian.
\end{proof}

\begin{remark}
    One may think of $\bar\lambda$ as a path from $1$ to $\lambda$ on $\Gm$, and the isomorphism $\psi_f(\CF)\isoto\psi_{\lambda f}(\CF)$ is induced by transporting the nearby fibre $f^{-1}(1)$ to $f^{-1}(\lambda^{-1})$ along the path $-\bar{\lambda}$ from $1$ to $\lambda^{-1}$.
\end{remark}

\begin{lemma}\label{lem_smoothqcqs}
    (1) Assume $\ell\neq p$. Let $g: Y\rightarrow X$, $f: X\rightarrow\AAA^1$ be maps of rigid analytic varieties with $g$ smooth, and $\CF\in D^{}(X)$. Then there is a canonical isomorphism $g^*\psi_f(\CF)\isoto\psi_{fg}(g^*\CF)$. Similarly for $\phi$.\\
    (2) Let $g: Y\rightarrow X$, $f: X\rightarrow\AAA^1$ be maps of rigid analytic varieties with $g$ quasi-compact quasi-separated, and $\CG\in D^{}(Y)$. Then there is a canonical isomorphism $g_*\psi_{fg}(\CG)\isoot\psi_{f}(g_*\CG)$. Similarly for $\phi$.
\end{lemma}

\begin{proof}
    This follows from standard diagram chase in the following diagram, using smooth base change, quasi-compact base change (\cite[\nopp 0.3.1, 4.1.1]{huber_etale_1996}), and the fact that $g_*$ commutes with colimits for quasi-compact quasi-separated $g$ (\cite[\nopp 9.1.2]{zavyalov_some_2024}). We omit the details.
\[\begin{tikzcd}
	{Y_0} & Y & {Y^\times_n} \\
	{X_0} & X & {X^\times_n} \\
	0 & {\mathbf{A}^1} & {\mathbf{G}_m}
	\arrow[hook, from=1-1, to=1-2]
	\arrow[from=1-1, to=2-1]
	\arrow["\lrcorner"{anchor=center, pos=0.125}, draw=none, from=1-1, to=2-2]
	\arrow["g", from=1-2, to=2-2]
	\arrow[from=1-3, to=1-2]
	\arrow["\lrcorner"{anchor=center, pos=0.125, rotate=-90}, draw=none, from=1-3, to=2-2]
	\arrow[from=1-3, to=2-3]
	\arrow[hook, from=2-1, to=2-2]
	\arrow[from=2-1, to=3-1]
	\arrow["\lrcorner"{anchor=center, pos=0.125}, draw=none, from=2-1, to=3-2]
	\arrow["f", from=2-2, to=3-2]
	\arrow[from=2-3, to=2-2]
	\arrow["\lrcorner"{anchor=center, pos=0.125, rotate=-90}, draw=none, from=2-3, to=3-2]
	\arrow[from=2-3, to=3-3]
	\arrow[hook, from=3-1, to=3-2]
	\arrow["{e_n}"', from=3-3, to=3-2]
\end{tikzcd}\]
\end{proof}

\begin{lemma}\label{lem_comparisonphi}
    Let $f: \CX\rightarrow\AAA^1$ be a map of finite type schemes over $K$, and $\CF\in D^b_c(\CX^\X)$. Let $f^{an}: X\rightarrow\AAA^1$ and $\CF^{an}$ be their analytifications. Then there is a canonical isomorphism $(\psi_f(\CF))^{an}\isoto\psi_{f^{an}}(\CF^{an})$. Similarly for $\phi$.
\end{lemma}

\begin{proof}
    This follows from the fact that $p_n^*, p_{n*}, j_*, i^*$ all commute with analytification (\cite[\nopp 3.16]{bhatt_six_2022}) and that $(-)^{an}: D(\CX_0)\rightarrow D(X_0) $, being a left adjoint, commutes with colimits. The map $(\psi_f(\CF))^{an}\isoto\psi_{f^{an}}(\CF^{an})$ is induced from the maps in these comparisons.
\end{proof}

\begin{example}\label{ex_phicomp}
    (1) Let $f: X\rightarrow\AAA^1$ be a smooth map of rigid analytic varieties, and $\CL$ a local system on $X$. Then $\phi_f(\CL)=0$.\\
    (2) Let $\CL$ be a local system in degree $0$ on the punctured disc $\Gm$, corresponding to a representation $\pi^{\mathrm{fin}}_1(\Gm,1)\rightarrow \mathrm{Aut}_K(\CL_1)$. Then there exists a canonical isomorphism $\psi_{id}(\CL)\simeq\CL_1$, with the monodromy action identified with the representation.
\end{example}

\begin{proof}
    (1) By induction on the amplitude, we may assume $\CL$ is concentrated in degree $0$. As $\phi$ is \etale local, we may assume $\CL$ is isomorphic to a constant sheaf of the form $\pi_X^*M$, where $\pi_X: X\rightarrow\Spa(K)$ and $M$ is a finite type $\Lambda$-module. For the rest of this proof, we denote $\pi_{(-)}^*M$ by $\underline{M}$. We may further assume $f: X\rightarrow\AAA^1$ is the projection $\AAA^n\rightarrow \AAA^1: (x_1,...,x_n)\mapsto x_1$ (use \cite[\nopp 1.6.10]{huber_etale_1996} to reduce to projection of polydiscs, then extend). Consider the diagram:
\[\begin{tikzcd}
	{\tilde{0}\times \mathbf{A}^{n-1}} & {\mathbf{A}^n} & {\mathbf{G}_m\times \mathbf{A}^{n-1}} \\
	{0\times \mathbf{A}^{n-1}} & {\mathbf{A}^n} & {\mathbf{G}_m\times \mathbf{A}^{n-1}}
	\arrow["{i'_a}", hook, from=1-1, to=1-2]
	\arrow["{p_{0a}}"', from=1-1, to=2-1]
	\arrow["\lrcorner"{anchor=center, pos=0.125}, draw=none, from=1-1, to=2-2]
	\arrow["{\bar{p}_a}", from=1-2, to=2-2]
	\arrow["{j'_a}"', hook', from=1-3, to=1-2]
	\arrow["\lrcorner"{anchor=center, pos=0.125, rotate=-90}, draw=none, from=1-3, to=2-2]
	\arrow["{p_a}", from=1-3, to=2-3]
	\arrow["i"', hook, from=2-1, to=2-2]
	\arrow["j", hook', from=2-3, to=2-2]
\end{tikzcd}\]
Here $a\in\ZZ_{\geq1}$ and $\bar{p}_a$ is $(x_1,...,x_n)\mapsto (x_1^a, x_2,...,x_n)$. Then, $\psi_f(\underline{M})\simeq\varinjlim_ai^*j_*p_{a*}\underline{M}\simeq\varinjlim_ai^*\bar{p}_{a*}j'_{a*}\underline{M}$ $\simeq\varinjlim_a\bar{p}_{0a*}i_a'^*j'_{a*}\underline{M}$. By purity (\cite[\nopp 3.9.1]{huber_etale_1996}), $\CH^r(\bar{p}_{0a*}i_a'^*j'_{a*}\underline{M})\simeq \underline{M}$ if $r=0$, $\underline{M}(-1)$ if $r=1$, and $0$ otherwise. The classes in $\CH^1$ are represented by Kummer torsors, so are killed by increasing $a$. So $\psi_f(\underline{M})\simeq\CH^0(\bar{p}_{0a*}i_a'^*j'_{a*}\underline{M})$. Take $a=1$, one sees that $\psi_f(\underline{M})\simeq i^*\underline{M}$ via the specialisation map. This finishes the proof.\\\\
    (2) Denote $j: \Gm\hookrightarrow\AAA^1$, $i: 0\hookrightarrow\AAA^1$, and $e_n: \Gm\rightarrow\Gm$, $z\mapsto z^n$. The morphism is constructed as follows: take the smallest $n\in\ZZ_{\geq1}$ such that $\CL$ is trivialised by the Kummer covering $e_n: \Gm\rightarrow\Gm$. There exists a unique isomorphism $\iota: e_n^*\CL\simeq e_n^*\underline{\CL_1}$ which restricts to $id$ at $1$ in the source $\Gm$. Then, ignoring monodromies, $\psi_{id}(\CL)\simeq\varinjlim_{n|m}i^*j_*e_{m*}e_m^*\CL\simeq\varinjlim_{n|m}i^*j_*e_{m*}e_m^*\underline{\CL_1}\simeq\psi_{id}(\underline{\CL_1})$. By (1), $\CL_1=i^*\underline{\CL_1}\xrightarrow{\mathrm{sp}}\psi_{id}(\underline{\CL_1})$ is an isomorphism. So we get $\CL_1\isoto\psi_{id}(\CL)$. The monodromy action on $\psi_{id}(\CL)$ factors through $\mathrm{Aut}(e_n)=\mu_n$ and is induced by $\sigma^*e_n^*\CL\isoto e_n^*\CL$, $\sigma\in\mu_n$. Under the isomorphism $\iota: e_n^*\CL\simeq e_n^*\underline{\CL_1}$ and restrict to the stalks at $e_n^{-1}(1)$, this is exactly the monodromy representation on $\CL_1$. One concludes that $\CL_1\isoto\psi_{id}(\CL)$ identifies the representation with the monodromy action.
\end{proof}

\section{Finiteness and a Milnor fibre interpretation}\label{sec_finiteness_Milnor}
    
In this section, we prove that nearby and vanishing cycles preserve Zariski-constructibility, and give a Milnor fibre interpretation of their stalks. The set-up is as in §\ref{sec_intro}, $\Lambda=\FF_{\ell^r}$ or $\ZZ/\ell^r$, with $\ell=p\neq0$ allowed until Remark \ref{rmk_nonZariskiConst_example} and $\ell\neq p$ for the rest.

\begin{proposition}\label{prop_phipreserveszc}
    Let $f: X\rightarrow\AAA^1$ be a map of rigid analytic varieties, and $\CF\in \DX$. Then $\psi_f(\CF)$, $\phi_f(\CF)$ lie in $D^{(b)}_{zc}(X_0)$.
\end{proposition}

\begin{proof}
    It suffices to prove for $\psi_f(\CF)$, the result for $\phi_f(\CF)$ follows. As Zariski-constructibility can be checked locally (\cite[\nopp 3.5]{bhatt_six_2022}), we may assume $X=\Spa(A)$ is affinoid. Then every $\CF\in D^b_{zc}(X)$ can be obtained by finitely many steps of taking cones, shifts, and summands of sheaves of the form $g_*\underline{M}$, where $g: Y\rightarrow X$ is a finite map and $\underline{M}$ is a constant constructible sheaf in degree $0$ on $Y$ (\cite[\nopp 3.6]{bhatt_six_2022}), so we may assume $\CF=g_*\underline{M}$. By Lemma \ref{lem_smoothqcqs}.2 and the fact that finite pushforward preserves Zariski-constructibility, we may further assume $\CF=\underline{M}$.\\\\
    We would like to show $\psi_f(\underline{M})$ is in $D^b_{zc}(X_0)$. By resolution of singularities (\cite[\nopp 1.1.13.(i), 1.1.11]{temkin_functorial_2018}), there exists a proper surjective map $a_0: X'\rightarrow X$, with $X'$ smooth and $(fa_0)^{-1}(0)=X_0\X_X X'$ strictly monomial. Denote $fa_0$ by $f'$. We first show $\psi_{f'}(\underline{M}_{X'})$ is Zariski-constructible. For each $s\in X_0'$, suppose $s$ is contained in $r$ irreducible components of the underlying simple normal crossings divisor $E$ of $X_0'$, there exists an open neighbourhood $s\in U\subseteq X'$ and functions $T_1,...,T_r,...,T_n\in \CO_U(U)$, such that $U\xrightarrow{(T_1,...,T_n)}\AAA^n$ is \etale and $X_0\cap U'=V(T_1^{n_1}...T_r^{n_r})$. Consequently $f'_U=uT_1^{n_1}...T_r^{n_r}$ for some unit $u$. Let $U'$ be the following fibre product.
\[\begin{tikzcd}
	{U'} & U \\
	{\mathbf{G}_m} & {\mathbf{G}_m}
	\arrow[from=1-1, to=1-2]
	\arrow["{u^{\frac{1}{n_1}}}"', from=1-1, to=2-1]
	\arrow["\lrcorner"{anchor=center, pos=0.125}, draw=none, from=1-1, to=2-2]
	\arrow["u", from=1-2, to=2-2]
	\arrow["{z\mapsto z^{n_1}}"', from=2-1, to=2-2]
\end{tikzcd}\]
    On $U'$, we have $f'_{U'}=T_1'^{n_1}T_2^{n_2}...T_r^{n_r}$ with $T_1'=u^{\frac{1}{n_1}}T_1$, and $U'\xrightarrow{(T_1',T_2,...,T_n)}\AAA^n$ is still \etale. As nearby cycles are \etale local, it is sufficient to prove $\psi_{T_1'^{n_1}...T_r^{n_r}}(\underline{M}_{\AAA^n})$ is Zariski-constructible. This follows from the comparison with the algebraic case (Lemma \ref{lem_comparisonphi}).\\\\
    Finally, we prove $\psi_f(\underline{M})$ is Zariski-constructible. Resolution of singularities gives a proper hypercovering $a_{\bullet}: X_{\bullet}\rightarrow X$ with each $X_i$ smooth and $X_i\X_X X_0$ strictly monomial. By cohomological descent, $\underline{M}_X\isoto \varprojlim_{\Delta}(a_{0*}\underline{M}_{X_0}\rightrightarrows a_{1*}\underline{M}_{X_1}\rightthreearrows\cdots)$, where $\varprojlim_{\Delta}$ denotes the totalisation. Apply $\psi_f$, we get (notations as in Definition \ref{def_van}):
\begin{align*}
\psi_f(\underline{M}_X) &\isoto\varinjlim i^*j_{n*}j_n^*\varprojlim_{\Delta}(a_{0*}\underline{M}_{X_0}\cdots) \\
&\simeq\varinjlim i^*\varprojlim_{\Delta}(j_{n*}j_n^*a_{0*}\underline{M}_{X_0}\cdots)\,\,(\text{$\varprojlim_{\Delta}$ commutes with pullbacks, $j_{n*}$ commutes with limits})\\
&\simeq\varprojlim_{\Delta}(\varinjlim i^*j_{n*}j_n^*a_{0*}\underline{M}_{X_0}\cdots)\,\,(\text{$\varprojlim_{\Delta}$ commutes with filtered colimits with coconnective entries})\\
&\simeq\varprojlim_{\Delta}(a_{0*}\psi_{fa_0}(\underline{M}_0)\rightrightarrows\cdots)\,\,(\text{$\psi$ commutes with proper pushforwards})
\end{align*}
As $\psi$ has finite cohomological dimension if the space is finite dimensional\footnote{Proof (see Lemma \ref{lem_j_coho_fini} for a more general statement): recall $\psi_f=\varinjlim i^*j_*p_{n*}p_n^*$, the only possibly non-obvious part is to bound the cohomological dimension of $j_*$. It suffices to show: for $X$ affinoid, and $\CF$ a torsion abelian sheaf on $X^\X$ in degree $0$, $R\Gamma(X^\X,\CF)\in D(\ZZ)$ is bounded uniformly in $\CF$. Write $X^\X=\cup_i X^\X_i$, where $X^\X_i=\{|f|\geq\epsilon_i\}$, and $\{\epsilon_i\}$ is a sequence in $|K^\X|$ converging to $0$. Note, each $X^\X_i$ is an affinoid open and this is an increasing union. By the bound on cohomological dimension (\cite[\nopp 2.8.3]{huber_etale_1996}), $R\Gamma(X^\X_i,\CF)\in D^{[0,2\dim(X)]}(\ZZ)$ . So $R\Gamma(X^\X,\CF)\simeq R\varprojlim_iR\Gamma(X^\X_i,\CF)\in D^{[0,2\dim(X)+1]}(\ZZ)$.}, $\psi_f(\underline{M}_X)$ is bounded. Each $\CH^i(\psi_f(\underline{M}_X))$ can be computed using a finite truncation of $\varprojlim_{\Delta}(a_{0*}\psi_{fa_0}(\underline{M}_0)\rightrightarrows\cdots)$. This is a finite limit, with each term Zariski-constructible, so is Zariski-constructible.
\end{proof}

\begin{remark}\label{rmk_nonZariskiConst_example}
    As $\psi_f(\CF)$ only depends on $\CF|_{X^\times}$, for $\CF\in D^{(b)}_{zc}(X^\times)$ which can be extended to $\DX$, we also have $\psi_f(\CF)\in D^{(b)}_{zc}(X_0)$. Note that $\psi_f(\CF)$ in general is not Zariski-constructible for $\CF\in D^{(b)}_{zc}(X^\times)$: consider the direct sum of the skyscraper sheaves $\underline{\Lambda}_{a_i}$ where $\{a_i\}_{i\in\NN}$ is a sequence of classical points in $D$ converging to $0$, then $\psi_{id}(\oplus_i\underline{\Lambda}_{a_i})\cong\varinjlim_n \frac{\Pi(\Lambda^{\oplus n})}{\oplus(\Lambda^{\oplus n})}$, where $\Pi$ and $\oplus$ are indexed by $\NN$.
\end{remark}

Assume $\ell\neq p$ for the rest of this section. The Milnor fibre interpretation of nearby and vanishing cycles will be a consequence of some finiteness results of Huber. For the convenience of the reader, we recall Huber's results here, adapted to our set-up.

\begin{definition}[{\cite[\nopp 1.2]{huber_finiteness_1998}}]
    Let $X$ be a rigid analytic variety. A sheaf $\CF\in D^{(+)}(X)$ is called \underline{overconvergent-quasi-constructible} (ocqc) if each $\CH^i(\CF)$ is ocqc. A sheaf $\CF$ in degree $0$ is ocqc if it is overconvergent and for every $x\in X$, there exist an \etale morphism $g: Y\rightarrow X$, a locally closed constructible subset $L\subseteq Y$ with $x\in g(L)$, and a decreasing sequence $Y=Y_0\supseteq Y_1\supseteq...\supseteq Y_r=\varnothing$ of Zariski-closed subspaces such that the restriction of $\CF$ on $L^\circ\cap(Y_i-Y_{i+1})$ extends to a locally constant sheaf in degree $0$ of finite type stalks on $L\cap(Y_i-Y_{i+1})$. Here, $L^\circ$ is the interior of $L\subseteq Y$.
\end{definition}

\begin{theorem}[{\cite[\nopp 2.1]{huber_finiteness_1998}}]\label{thm_huber1}
    Let $f: X\rightarrow Y$ be a quasi-compact quasi-separated map of rigid analytic varieties, with $\dim Y\leq 1$. Then, $f_*$ preserves ocqc sheaves.
\end{theorem}

\begin{theorem}[{\cite[\nopp 3.6]{huber_finiteness_1998}}]\label{thm_huber2}
    Let $Z\hookrightarrow X$ be a closed imbedding of rigid analytic varieties. Assume $X=\Spa(A)$ is affinoid, and $\{f_1,...,f_r\}\subseteq A$ be such that $Z=V(f_1,...,f_r)$ set-theoretically. For $\epsilon\in|K^\X|$, denote $T_\epsilon=\{|f_i|\leq\epsilon, \forall i\}$, and $T'_\epsilon=\{|f_i|<\epsilon, \forall i\}$. Note $T_\epsilon$ and $X-T'_\epsilon$ (resp. $T'_\epsilon$ and $X-T_\epsilon$) are quasi-compact open (resp. closed) subsets of $X$. Then, for $\CF\in D^b(X)$ ocqc, $\exists\,\epsilon_0\in|K^\X|$ such that $\forall\,\epsilon<\epsilon_0$ with $\epsilon\in|K^\X|$, we have the following isomorphisms induced by restrictions:\\ 
    (1) $R\Gamma(T_\epsilon,\CF)\isoto R\Gamma(T'_\epsilon,\CF)\isoto R\Gamma(T'^\circ_\epsilon,\CF)\isoto R\Gamma(Z,\CF)$;\\
    (2) $R\Gamma(X-Z,\CF)\isoto R\Gamma(X-T'_\epsilon,\CF)\isoto R\Gamma(X-T^\circ_\epsilon,\CF)\isoto R\Gamma((X-T^\circ_\epsilon)^\circ,\CF)$.\\
    Here, $(-)^\circ$ denotes interior, and $R\Gamma(T'_\epsilon,\CF)$ and $R\Gamma(X-T^\circ_\epsilon,\CF)$ are cohomologies of pseudo-adic spaces. 
\end{theorem}

\begin{example}\label{ex_huber}
    (1) Zariski-constructible sheaves are ocqc.\\
    (2) By \cite[\nopp 1.3.(iv)]{huber_finiteness_1998}, on a $1$-dimensional quasi-compact rigid analytic variety $X$, an overconvergent sheaf $\CF$ in degree $0$ is ocqc if and only if there is a quasi-compact open $U\subseteq X$ such that $\CF|_U$ is Zariski-constructible, and, for every $x\in X-U$, there is a locally closed locally constructible subset $ L\subseteq X$ with $x\in L$, such that $\CF|_{L^\circ}$ extends to a locally constant sheaf in degree $0$ of finite type stalks on $L$. This implies: for a maximal point $x\in X$, if $x$ is a classical point, then there exists some open neighbourhood $V$ of $x$ such that $\CF$ is locally constant on $V-x$ (possibly also on $V$); otherwise, $\CF$ is locally constant in some open neighbourhood of $x$. We leave this as an exercise in point-set topology.
\end{example}

\begin{remark}\label{rmk_huber}
    Note the following consequence of Theorem \ref{thm_huber2}: let $X$ be a rigid analytic variety and $\CF\in \DX$. Then, at every classical point $x\in X$, if $\{f_1,...,f_r\}$ are local functions such that $x=V(f_1,...,f_r)$ set-theoretically, then $R\Gamma(T_\epsilon,\CF)\isoto\CF_x$ for every small enough $\epsilon\in|K^\X|$.
\end{remark}

\begin{proposition}\label{prop_nearbyfibinterp}
    Let $f: X=\Spa(A)\rightarrow\AAA^1$ be a map of rigid analytic varieties, $\CF\in D^b_{zc}(X)$, $x\in X_0$ be a classical point, and $f_1,...,f_r\in A$ such that $x=V(f_1,...,f_r)$ set-theoretically. Then there exists $\epsilon_0\in |K^\times|$ such that for every $0<\epsilon<\epsilon_0$ with $\epsilon\in|K^\times|$, there exists $\eta_0\in |K^\times|$ such that for every $0<\eta<\eta_0$ with $\eta\in|K^\times|$, and every classical point $a\in D^\times(\eta)$, there exists an isomorphism $\psi_f(\CF)_x\cong R\Gamma(U_{\epsilon,a},\CF)$, where $U_{\epsilon,a}:=\{f=a,|f_i|\leq\epsilon,\forall i\}\subseteq X$ is the “Milnor fibre”. Similarly for $\phi_f(\CF)_x$: there exist (possibly different) constants $(\epsilon_0,\epsilon,\eta_0,\eta)$ as above such that there exists an isomorphism $\phi_f(\CF)_x\cong \cone(\CF_x\rightarrow R\Gamma(U_{\epsilon,a},\CF))$ for every classical point $a\in D^\times(\eta)$.
\end{proposition}

\begin{proof}
    We prove for $\psi_f$, the result for $\phi_f$ follows from a similar argument. As $\psi_f(\CF)$ is Zariski-constructible (Proposition \ref{prop_phipreserveszc}), by Remark \ref{rmk_huber} there exists $\epsilon_0\in|K^\X|$ such that for every $0<\epsilon<\epsilon_0$ with $\epsilon\in|K^\X|$, we have the restriction isomorphism $\psi_f(\CF)_x\isoot R\Gamma(U_{\epsilon,0},\psi_f(\CF))$. By Lemma \ref{lem_smoothqcqs}.2, $R\Gamma(U_{\epsilon,0},\psi_f(\CF))\simeq\psi_{id}(f_*(\CF|_{U_\epsilon}))$, where $U_\epsilon:=\{|f_i|\leq \epsilon, \forall\,i\}\subseteq X$. By Theorem \ref{thm_huber1} and Example \ref{ex_huber}.2, $f_*(\CF|_{U_\epsilon})$ is locally constant on some small punctured disc $D^\X(\eta)$, so, once a (classical) base point $a\in D^\X(\eta)(K)$ is chosen, Example \ref{ex_phicomp}.2 applies\footnote{In Example \ref{ex_phicomp}.2 we only discussed the case when $\CL$ is in degree 0 and lives on $\Gm$. For a general local system $\CL$ on $\Gm$, a similar argument gives an isomorphism $\psi_{id}(\CL)\cong \CL_1$ without canonicity and monodromy assertions. For a general local system $\CL$ on a punctured disc, one can apply the same arguments to a sufficiently small punctured disc (by \cite{lutkebohmert_riemanns_1993}, $\CL$ is trivialised by a Kummer covering in a sufficiently small punctured disc) and obtain an isomorphism $\psi_{id}(\CL)\cong \CL_a$. This suffice for our purpose here. Canonicity and monodromy assertions can make sense but requires developing the notion of monodromic sheaves; this will be done in \cite{zhou_microlocal_2025}.} and gives an isomorphism $\psi_{id}(f_*(\CF|_{U_\epsilon}))\cong (f_*(\CF|_{U_\epsilon}))|_a\simeq R\Gamma(U_{\epsilon,a},\CF)$. Combined, we get $\psi_f(\CF)_x\isoot R\Gamma(U_{\epsilon,a},\CF)$.
\end{proof}

\begin{remark}
    (1) The same method proves that, given $(\epsilon_0,\epsilon,\eta_0,\eta)$ as above, there exists $\epsilon_0'\in|K^\X|$ such that for every $0<\epsilon'<\epsilon_0'$ with $\epsilon'\in|K^\times|$, we have $R\Gamma(U_{\epsilon,a,\leq\epsilon'}, \CF)\isoto R\Gamma(U_{\epsilon, \eta}, \CF)\cong\psi_f(\CF)_x$, where $U_{\epsilon,\eta,\leq\epsilon'}:=\{|f-a|\leq\epsilon',|f_i|\leq\epsilon,\forall i\}$ is the “Milnor tube”.\\
    (2) In the algebraic context, the Milnor fibre is $\CX_{(x),\bar{\eta}}=\CX_{(x)}\times_{\AAA^1_{(0)}}\bar{\eta}$, and $\psi_f(\CF)_x\simeq R\Gamma(\CX_{(x),\bar{\eta}},\CF)$ essentially by definition. In the complex analytic context, the analogue of this Proposition is non-trivial and well-known, see for example, \cite[248, 371]{schurmann_topology_2003}.
\end{remark}

\section{Beilinson's construction and gluing}\label{sec_beilinson_gluing}

In this section, we discuss Beilinson's construction of the unipotent nearby and vanishing cycles and gluing, in our context. We first prove some preparational results on perverse sheaves, including a General Artin-Grothendieck Vanishing. The set-up is as in §\ref{sec_intro}, $\Lambda=\FF_{\ell^r}$ or $\ZZ/\ell^r$, with $\ell\neq p$.

\begin{definition} [\textit{c.f.} {\cite[\nopp 3.4]{hansen_vanishing_2020}}]
    A rigid analytic variety $X$ is called \underline{weakly Stein} if there exist affinoid opens $U_0\subseteq U_1\subseteq ...\subseteq X$, $i\in\NN$, such that $X=\cup_i U_i$.
\end{definition}

\begin{lemma}[\textit{c.f.} {\cite[\nopp 3.4]{hansen_vanishing_2020}}]\label{lem_wSteinartinvan}
    Let $X$ be a weakly Stein rigid analytic variety, and $\CF\in \mathrm{Sh}_{zc}(X)$. Then $R\Gamma(X,\CF)\in D^{[0,\mathrm{dim}(X)]}(\Lambda)$.
\end{lemma}

\begin{proof}
    By Affinoid Artin-Grothendieck Vanishing (\cite[\nopp 7.3]{bhatt_arc_2021}), the result holds for $X$  affinoid. Write $X=\cup_i U_i$ as an increasing union of affinoid opens. Then $R\Gamma(X,\CF)\simeq R\varprojlim_i R\Gamma(U_i,\CF)\in  D^{[0,\mathrm{dim}(X)+1]}(\Lambda)$. The “$\mathrm{dim}(X)+1$” comes from possible non-vanishing of $R^1\varprojlim$. But each $R^q\Gamma(U_i,\CF)$ is in fact finite (\cite[\nopp 3.1]{huber_finiteness_1998}), so the system is Mittag-Leffler and $R^1\varprojlim$ vanishes.
\end{proof}

\begin{lemma}[\textit{c.f.} {\cite[\nopp 4.1.5, 4.1.6]{BBDG}}]\label{lem_perversecriterion}
    Let $X$ be a rigid analytic variety, and $\CF\in\DX$. Then the following are equivalent:\\
    (1) $\CF\in$ $^p\!D^{\leq 0}_{zc}(X)$,\\
    (2) for every affinoid open $U\subseteq X$, have $R\Gamma(U,\CF)\in D^{\leq 0}(\Lambda)$,\\
    (3) for every weakly Stein open $U\subseteq X$, have $R\Gamma(U,\CF)\in D^{\leq 0}(\Lambda)$.
\end{lemma}

\begin{proof}
    The same argument as in Lemma \ref{lem_wSteinartinvan} shows (2)$\implies$(3). (3)$\implies$(2) is clear. We now show (1)$\implies$(2) and (2)$\implies$(1).\\ 
    (1)$\implies$(2): it suffices to show this for $X$ affinoid and $U=X$. We do induction on the amplitude (with respect to the usual t-structure) of $\CF$. If $\CF=0$, this is clear. For a general $\CF$, suppose $a$ is the smallest integer such that $\CH^a(\CF)\neq0$. Consider $\tau_{\leq a}\CF\rightarrow\CF\rightarrow\tau_{>a}\CF\rightarrow$, note each term is in $^p\!D^{\leq 0}_{zc}(X)$. Take cohomology: $R\Gamma(X,\tau_{\leq a}\CF)\rightarrow R\Gamma(X,\CF)\rightarrow R\Gamma(X,\tau_{>a}\CF)\rightarrow$. As $\dim\mathrm{supp}(\tau_{\leq a}\CF)\leq -a$, by Affinoid Artin-Grothendieck Vanishing, we have $R\Gamma(X,\tau_{\leq a}\CF)\in D^{\leq0}(X)$. $R\Gamma(X,\tau_{>a}\CF)$ also lies in $D^{\leq0}(X)$ by induction hypothesis. So $R\Gamma(X,\CF)\in D^{\leq0}(X)$.\\\\
    (2)$\implies$(1): the original proof in \cite[\nopp 4.1.7, 4.1.8]{BBDG} applies in our context, except: (a) in the statement of Lemme 4.1.7 in \textit{loc. cit.}, replace the word “affine” in the third line by “weakly Stein”; (b) in the first paragraph of 4.1.8 in \textit{loc. cit.}, the existence of $F_i$'s in our context follows from Lemma \ref{lem_stratification}; (c) one uses Lemma \ref{lem_wSteinartinvan} as needed.
\end{proof}

\begin{lemma}\label{lem_stratification}
    Let $X$ be an affinoid rigid analytic variety over $K$ of dimension $d$, and $\CF\in\DX$. Then, there exist Zariski-closed subspaces $\varnothing=Z_{-1}\subseteq Z_0\subseteq...\subseteq Z_d=X$ such that each $Z_i-Z_{i-1}$ is smooth, of pure dimension $i$, and $\CF$ is lisse on $Z_i-Z_{i-1}$.
\end{lemma}

We will freely use the following facts: (1) every union of irreducible components of a rigid analytic variety is Zariski-closed (\cite[\nopp 2.2.4.3]{conrad_irreducible_1999}); (2) the non-smooth locus of a reduced rigid analytic variety over an algebraically closed field is Zariski-closed and nowhere dense (discussion after 3.3.1 in \textit{loc. cit.}).

\begin{proof}
    Let $X'$ be the union of dimension $d$ irreducible components of $X$. Its smooth locus $X'_{sm}$ is a dense Zariski open. As $\CF$ is Zariski-constructible, there exists a dense Zariski open $U\subseteq X'$ on which $\CF$ is lisse. Let $Z_{d-1}'=\{\text{irreducible components of $X$ not in $X'$}\}\cup\{X'-(U\cap X'_{sm})\}$. This is Zariski closed. If $\dim Z_{d-1}'=d-1$, set $Z_{d-1}=Z_{d-1}'$; otherwise, set $Z_{d-1}=Z_{d-1}'\cup H$, where $H$ is any hypersurface in $X'$ (which exists, as $X'$ is affinoid). Iterate. 
\end{proof}

We get an analogue of General Artin-Grothendieck Vanishing, this is \cite[\nopp 4.1.1]{BBDG} in the algebraic context, and \cite[\nopp 10.3.17]{kashiwara_sheaves_1990} in the complex analytic context.

\begin{proposition}[General Artin-Grothendieck Vanishing]\label{prop_artinvanishing}
    Let $f: X\rightarrow Y$ be a map of rigid analytic varieties such that for every affinoid open $V\subseteq Y$, $f^{-1}(V)$ is weakly Stein. Then, $f_*$ (resp. $f_!$) is perverse right (resp. left) t-exact, in the following sense: for every $\CF\in$ $^p\!D^{\leq 0}_{zc}(X)$ (resp. $^p\!D^{\geq 0}_{zc}(Y)$) such that $f_*\CF$ (resp. $f_!\CF$) lies in $D^{(b)}_{zc}(Y)$, we have $f_*\CF\in$ $^p\!D^{\leq 0}_{zc}(Y)$ (resp. $f_!\CF\in$ $^p\!D^{\geq 0}_{zc}(Y)$).
\end{proposition}

\begin{proof}
    It suffices to prove for $f_*$, the result for $f_!$ follows by duality\footnote{Use (1) from the following formal fact (which we will use without repeated mention): let $f: X\rightarrow Y$ be a map of rigid analytic varieties, then there are canonical isomorphisms: (1) $f_*\DD\CF\simeq\DD f_!\CF$, $\forall\,\CF\in D(X)$; (2) $f_!\DD\CF\simeq\DD f_*\CF$, $\forall\,\CF\in D(X)$ such that $\CF$ and $f_!\DD\CF$ are reflexive (recall that a sheaf $\CH$ is \underline{reflexive} if the canonical map $\CH\rightarrow\DD\DD\CH$ is an isomorphism); (3) $f^*\DD\CG\simeq\DD f^!\CG$, $\forall\,\CG\in D(Y)$ reflexive; (4) $f^!\DD\CG\simeq\DD f^*\CG$, $\forall\,\CG\in D(Y)$.}. We first show $f_*\CF\in$ $^p\!D^{\leq 0}_{zc}(Y)$, by criterion Lemma \ref{lem_perversecriterion}.2, it suffices to show, for every affinoid open $U\subseteq Y$, that $R\Gamma(U,f_*\CF)\in D^{\leq 0}(\Lambda)$. Smooth base change implies $R\Gamma(U,f_*\CF)\simeq R\Gamma(f^{-1}U,\CF)$. By assumption, we may write $f^{-1}(U)$ as an increasing union of affinoid opens $V_i\subseteq X$. We have $R\Gamma(V_i,\CF)\in D^{\leq0}$ by Lemma \ref{lem_perversecriterion}.2. Then, a similar $R\varprojlim$ argument as in Lemma \ref{lem_wSteinartinvan} shows $R\Gamma(f^{-1}U,\CF)\in D^{\leq0}$.
\end{proof}

\begin{example}\label{ex_j_exact}
    Let $f: X\rightarrow\AAA^1$ be a map of rigid analytic varieties, then $j: X^\X\rightarrow X$ satisfies the condition in Proposition \ref{prop_artinvanishing}. Indeed, for every affinoid open $U=\Spa(A)\subseteq X$, $j^{-1}(U)$ is the increasing union of the affinoids $\{x\in\Spa(A)\,|\,|f|\geq\epsilon\}$ as $\epsilon\in|K^\X|$ goes to $0$. So $j_*$ (resp. $j_!$) is perverse right (resp. left) t-exact, hence perverse t-exact (combining the other direction proved in \cite[\nopp 4.2.2]{bhatt_six_2022}).
\end{example}

\begin{terminology}\label{terminology_extendable}
    Let $f: X\rightarrow\AAA^1$ be a map of rigid analytic varieties. We call an $\CF\in \Perv(X^\X)$ \underline{extendable} if $\CF$ can be extended to an object of $\Dbzc(X)$. The full subcategory of extendable objects is denoted by $\Perv_{ext}(X^\X)$. Similarly for $\CF\in D^{(b)}_{zc}(X^\X)$.
\end{terminology}

We thus obtain the following analogue of \cite[\nopp 4.1.10, 4.1.12]{BBDG}, proved in the same way.

\begin{corollary}\label{cor_4110}
    Let $f: X\rightarrow\AAA^1$ be a map of rigid analytic varieties. Denote $j: X^\X\hookrightarrow X$, $i: X_0\hookrightarrow X$. Let $\CF$ (resp. $\CG$) be a perverse sheaf on $X$ (resp. an extendable perverse sheaf on $X^\X$). Then,\\
    (1) $i^*\CF\in$ $^p\!D^{[-1,0]}_{zc}(X_0)$, $i^!\CF\in$ $^p\!D^{[0,1]}_{zc}(X_0)$, and we have the following exact sequences:
    \begin{align*}
    0\rightarrow i_*\,^p\CH^{-1}i^*\CF\rightarrow j_!j^*\CF\rightarrow\CF\rightarrow i_*\,^p\CH^0i^*\CF\rightarrow0\\
    0\rightarrow i_*\,^p\CH^{0}i^!\CF\rightarrow \CF\rightarrow j_*j^*\CF\rightarrow i_*\,^p\CH^1i^!\CF\rightarrow0
    \end{align*}
    (2) $i^*j_{!*}\CG[-1]\simeq\,^p\CH^{-1}i^*j_*\CG\simeq\,^p\CH^0i^!j_!\CG$ and $i^!j_{!*}\CG[1]\simeq\,^p\CH^0i^*j_*\CG\simeq\,^p\CH^1i^!j_!\CG$. In particular, they are perverse. 
\end{corollary}

For the rest of this section, fix a topological generator $t$ of $\pi^{\mathrm{fin}}_1(\Gm,1)$. We revisit the decomposition $(-)\simeq (-)^P\oplus\!^P\!(-)$ in Lemma \ref{lem_iwasawa_basics}.3 and Remark \ref{rmk_t_non_t_decomposition}. Until the end of Corollary \ref{cor_u_nu_decomp}, we allow $\ell=p$. Given a finite $\Lambda$-module $M$ with a continuous $\pi^{\mathrm{fin}}_1(\Gm,1)$-action, consider the following alternative decomposition: by finiteness, $M$ (viewed as a $\Lambda[t]$-module) is supported at finitely many closed points of $\mathrm{Spec}(\Lambda[t])$, so has a canonical decomposition $\oplus_{\mathfrak{m}}M[\mathfrak{m}^\infty]\isoto M$, where $\mathfrak{m}$ ranges through all maximal ideals of $\Lambda[t]$, and $[\mathfrak{m}^\infty]$ denotes the $\mathfrak{m}^\infty$-torsion part. Let $M^u=$ $M[(t-1)^\infty]$ and $M^{nu}=\oplus_{\mathfrak{m}\neq (t-1)}M[\mathfrak{m}^\infty]$, we get $M\cong M^u\oplus M^{nu}$, where $t-1$ acts nilpotently on $M^u$ and invertibly on $M^{nu}$.

\begin{lemma}
    With notations as above, we have $M^P=M^u$ and $^P\!M=M^{nu}$.
\end{lemma}

\begin{proof}
    We will show $t-1$ acts nilpotently on $M^P$ and invertibly on $^P\!M$. It is then a simple exercise to see $M^P=M^u$ and $^P\!M=M^{nu}$. The $\hat{\ZZ}(1)$-action on $M^P$ factors through $\ZZ_\ell(1)$. So $M^P$ is naturally a $\Lambda[[\ZZ_\ell(1)]]$-module. It is well-known that $\Lambda[[\ZZ_\ell(1)]]\cong\Lambda[[\gamma]]$, with $\gamma$ being the image of $t-1$ in $\ZZ_\ell(1)$. By finiteness, for each $m\in M$, there exist $i$ and $s\in\NN$ such that $\gamma^im=\gamma^{i+s}m$, so $\gamma^i(1-\gamma^s)m=0$. As $1-\gamma^s$ is invertible in $\Lambda[[\gamma]]$, we get $\gamma^im=0$. By finiteness again, there exists $i\in \NN$ such that $\gamma^iM=0$, \textit{i.e.}, $t-1$ acts nilpotently on $M^P$.\\\\
    Consider $M^P\hookrightarrow M\twoheadrightarrow M_P$, the $P$-coinvariant $M_P$ is isomorphic to $M^P$ via $\kappa: M_P\rightarrow M^P, x\mapsto \int_{\sigma\in P}\sigma x d\sigma$. So $M_P\simeq (M^P)_P\oplus(^P\!M)_P\simeq M_P\oplus (^P\!M)_P$, hence $(^P\!M)_P=0$. Similarly, $(^P\!M)^P=0$. These imply $(^P\!M)_\mu=0$ and $(^P\!M)^\mu=0$, so $t-1$ acts invertibly on $^P\!M$. 
\end{proof}

\begin{corollary}\label{cor_u_nu_decomp}
    We allow $\ell=p$. Let $f: X\rightarrow\AAA^1$ be a map of rigid analytic varieties, $\CF\in\DX$, and $\psi_f(\CF)\simeq \psi^P_f(\CF)\oplus\!^P\!\psi_f(\CF)$ be the canonical decomposition. Then $t-1$ acts nilpotently on $\psi^P_f(\CF)$ and invertibly on $^P\!\psi_f(\CF)$. We will also denote the decomposition as $\psi_f(\CF)\simeq\psi^u_f(\CF)\oplus\psi^{nu}_f(\CF)$. Similarly, we have $\phi_f(\CF)\simeq \phi^u_f(\CF)\oplus\phi^{nu}_f(\CF)(=\phi^P_f(\CF)\oplus\!^P\!\phi_f(\CF))$. Furthermore, $\mathrm{can}$ induces an isomorphism $\psi^{nu}_f(\CF)\isoto\phi^{nu}_f(\CF)$, and we have distinguished triangles $i^*\CF\xrightarrow{\mathrm{sp}}\psi_f^u(\CF)\xrightarrow{\mathrm{can}}\phi_f^u(\CF)\rightarrow$, $\phi^u_f(\CF)\xrightarrow{\mathrm{var}}\psi^u_f(\CF)\miwa\xrightarrow{\mathrm{cosp}}i^!\CF[2]\rightarrow$,  $i^*j_*j^*\CF\xrightarrow{\mathrm{sp}}\psi_f^u(\CF)\xrightarrow{\iota}\psi_f^u(\CF)\miwa\rightarrow$.
\end{corollary}

\begin{proof}
    This is a direct consequence of the lemma above and the finiteness of local sections, which is a consequence of Proposition \ref{prop_phipreserveszc}. The “Furthermore” part follows from the fact that $\mu$ acts trivially on $i^*(-)$.
\end{proof}

We now discuss Beilinson's construction in our context, following the exposition in \cites{morel_beilinson_2018}. Using results established so far, most of \cites{morel_beilinson_2018} carries through. In the following, we will only elaborate when changes are needed. The main differences are: (1) as we are working with coefficient $\Lambda=\FF_{\ell^r}$ or $\ZZ/\ell^r$, the logarithm of monodromy does not make sense. Instead, one should replace $N$ in \textit{loc. cit.} by $t-1$ throughout; (2) if $X_0$ is not quasi-compact, then the order of nilpotency of $t-1$ acting on $\psi_f^u(\CF)$ is only locally finite. Effectively, this means one either assumes $X$ quasi-compact or adds $\varinjlim$ when stating certain results.\\

Fix a map $f: X\rightarrow\AAA^1$ of rigid analytic varieties. Denote $\psi_f[-1]$ (resp. $\phi_f[-1]$) by $\Psi_f$ (resp. $\Phi_f$). Recall we have fixed a topological generator $t$ of $\pi^{\mathrm{fin}}_1(\Gm,1)$.

\begin{lemma}[{\cite[\nopp 1.3]{morel_beilinson_2018}}]\label{lem_psiuniistexact}
    $\Psi_f^u: \Dbzc(X)\rightarrow D^{(b)}_{zc}(X_0)$ is perverse t-exact.
\end{lemma}

\begin{proof}
    Let $\CF\in$ $^p\!D^{\leq 0}_{zc}(X)$, we want to show $\Psi_f^u(\CF)\in$ $^p\!D^{\leq 0}_{zc}(X_0)$. We have $\Psi_f^u(\CF)\xrightarrow{t-1}\Psi_f^u(\CF)\rightarrow i^*j_*j^*\CF\rightarrow$. The perverse cohomology long exact sequence associated to $j_!j^*\CF\rightarrow\CF\rightarrow i_*i^*j_*j^*\CF\rightarrow$ and Example \ref{ex_j_exact} imply $i^*j_*j^*\CF\in$ $^p\!D^{\leq 0}_{zc}(X_0)$. On each quasi-compact open in $X_0$, combine this with the perverse cohomology long exact sequence associated to $\Psi_f^u(\CF)\xrightarrow{t-1}\Psi_f^u(\CF)\rightarrow i^*j_*j^*\CF\rightarrow$ and the nilpotency of $t-1$ acting on $^p\CH^q(\Psi_f^u(\CF))$, we get $^p\CH^q(\Psi_f^u(\CF))=0, \forall\,q>0$. So $\Psi_f^u(\CF)\in$ $^p\!D^{\leq 0}_{zc}(X_0)$. The other direction is similar.
\end{proof}

For every $a\in\NN$, let $\CL_a$ be the rank $a+1$ local system on $\Gm$ with stalk $\Lambda\oplus\Lambda(-1)\oplus\cdots\oplus\Lambda(-a)$ at $1$ and monodromy $t$ represented by the unipotent Jordan block 
\[
\begin{pmatrix}
1      & 1 & 0      & \cdots & 0 \\
0      & 1 &        &        & \vdots \\
\vdots &   & \ddots &        & 0 \\
0      &   &        & 1      & 1 \\
0      & 0 & \cdots & 0      & 1
\end{pmatrix}
\]
For $a\leq b$, the injection to the first $a$ coordinates (resp. surjection to the last $a$ coordinates) defines a map $\alpha_{a,b}:\CL_a\hookrightarrow\CL_b$ (resp. $\beta_{b,a}:\CL_a\twoheadleftarrow\CL_b(a-b)$).

\begin{lemma}[{\cite[\nopp 3.2, 4.3]{morel_beilinson_2018}}]\label{lem_beilinsonpsi}
    Let $\CF\in \Perv_{ext}(X^\X)$. Then, there are natural isomorphisms: $i_*\Psi_f^u(\CF)\isoto\varinjlim_a\Ker(j_!(\CF\otimes\CL_a)\rightarrow j_*(\CF\otimes\CL_a))$ and $\Psi_f^u(\CF)\isoto\varinjlim_ai^*j_*(\CF\otimes\CL_a)[-1]$.
\end{lemma}

The colimits are taken in $D(X)$. They land in $\Perv(X_0)$ because the same proofs in \textit{loc. cit.} show this is the case locally.

\begin{proof}
    In the proof of Proposition 3.1 in \textit{loc. cit.}, replace $N$ by $t$ and note it acts on $\Psi^u_f\CF\otimes L_a$ by $t\otimes t$. Change the map $\Psi^u_f\CF\rightarrow\Psi^u_f\CF\otimes L_a$ to $x\mapsto(x,-\frac{(t-1)}{t}x,\frac{(t-1)^2}{t^2}x,...,(-1)^a\frac{(t-1)^a}{t^a}x)$.
\end{proof}

\begin{lemma}[{\cite[\nopp 4.2]{morel_beilinson_2018}}]\label{lem_beilinsonpsiduality}
    Let $\CF\in \Perv_{ext}(X^\X)$. Then, there is a natural isomorphism $\Psi_f^u(\DD\CF)(-1)\isoto\DD\Psi_f^u(\CF)$.
\end{lemma}

\begin{definition-lemma}[{\cite[\nopp 5.1]{morel_beilinson_2018}}]
    The \underline{maximal extension functor} $\Xi_f: \Perv_{ext}(X^\X)\rightarrow \Perv(X)$ is defined by $\CF\mapsto \varinjlim_a\Ker(\gamma_{a,a-1})$, where $\gamma_{a,a-1}$ is the composition $j_!(\CF\otimes f^*\CL_a)\xrightarrow{\beta_{a,a-1}}j_!(\CF\otimes f^*\CL_{a-1})(-1)\rightarrow j_*(\CF\otimes f^*\CL_{a-1})(-1)$.
\end{definition-lemma}

\begin{lemma}[{\cite[\nopp 5.5]{morel_beilinson_2018}}]\label{lem_Xi}
     There is a natural isomorphism $\DD\Xi_f\simeq\Xi_f\DD$, $\Xi_f$ is exact, and we have the following canonical maps and exact sequences, which are Verdier dual to each other:
     \begin{equation*}
     0\rightarrow j_!\xrightarrow{\delta}\Xi_f\xrightarrow{\alpha} i_*\Psi_f^u(-1)\rightarrow0\,, \quad\quad 0\rightarrow i_*\Psi_f^u\xrightarrow{\beta}\Xi_f\xrightarrow{\epsilon} j_*\rightarrow0
     \end{equation*}
\end{lemma}

\begin{definition}
    The functor $\Phi_f^{B}: \Perv(X)\rightarrow\Perv(X_0)$ is defined as $\CF\mapsto i^*H^0(j_!j^*\CF\xrightarrow{\delta+\mathrm{adj}}\Xi_fj^*\CF\oplus\CF\xrightarrow{\epsilon-\mathrm{adj}}j_*j^*\CF)$, with $\Xi_fj^*\CF$ placed in degree $0$. The maps $can'$ and $var'$ are defined as the maps induced by $i_*\Psi_f^uj^*\CF\xrightarrow{\beta}\Xi_fj^*\CF$ and $\Xi_fj^*\CF\xrightarrow{\alpha} i_*\Psi_f^uj^*\CF(-1)$.
\end{definition}

\begin{remark}
    We leave the question of the compatibility of $(\Phi_f^{B}$, $\mathrm{can'}$, $\mathrm{var'})$ with $(\Phi_f^u$, $\mathrm{can}$, $\mathrm{var}(t))$ from Corollary \ref{cor_u_nu_decomp} to the interested reader.
\end{remark}

\begin{lemma}[{\cite[\nopp 6.1]{morel_beilinson_2018}}]\label{lem_phiuniistexact}
    There is a natural isomorphism $\DD\Phi_f^B\simeq\Phi_f^B\DD$, $\Phi_f^B$ is perverse t-exact, and $\Psi_f^u\xrightarrow{\mathrm{can}}\Phi_f^B$ and $\Phi_f^B\xrightarrow{\mathrm{var'}}\Psi_f^u(-1)$ are Verdier dual to each other.
\end{lemma}

One may now apply the original proof in \cite[\nopp 3.1]{beilinson_how_1987} to obtain Beilinson's Gluing in our context:

\begin{proposition}[{\cite[\nopp 3.1]{beilinson_how_1987}}, Beilinson's Gluing]\label{prop_beilinsonglue}
    Let $f: X\rightarrow\AAA^1$ be a map of rigid analytic varieties. Let $GD(X,f)$ be the category of gluing data, defined as follows: the objects are quadruples $(\CF, \CG, u, v)$ where $\CF\in \Perv_{ext}(X^\X)$, $\CG\in \Perv(X_0)$, and $u: \Psi^u_f(\CF)\rightarrow\CG$, $v: \CG\rightarrow\Psi^u_f(\CF)(-1)$ are maps such that $vu=t-1$; the morphisms are the obvious ones.\\
    Then, $\mathrm{Perv}(X)$ and $GD(X,f)$ are equivalent via the functors $F: \mathrm{Perv}(X)\rightarrow GD(X,f)$, $\CF\mapsto (\CF|_{X^\times},\Phi^B_f(\CF), can', var')$, and $G: GD(X,f)\rightarrow\mathrm{Perv}(X)$, $(\CF, \CG, u, v)\mapsto H^0(i_*\Psi^u_f(\CF)\xrightarrow{\alpha+u} \Xi_f(\CF)\oplus i_*\CG\xrightarrow{{\beta-v}} i_*\Psi^u_f(\CF)(-1))$, where $\Xi_f(\CF)\oplus i_*\CG$ is placed in degree $0$ and $\alpha$ and $\beta$ are as in Lemma \ref{lem_Xi}.
\end{proposition}

\section{Perversity and duality}\label{sec_perversity_duality}

In this section, we discuss perverse t-exactness and duality of nearby and vanishing cycles. The set-up is as in §\ref{sec_intro}, $\Lambda=\FF_{\ell^r}$ or $\ZZ/\ell^r$, with $\ell\neq p$. Recall we denote $\psi_f[-1]$ (resp. $\phi_f[-1]$) by $\Psi_f$ (resp. $\Phi_f$).

\begin{theorem}\label{thm_psi_t_exact_duality}
    Let $f: X\rightarrow\AAA^1$ be a map of rigid analytic varieties.\\
    (1) Assume $\Lambda=\FF_{\ell^r}$. Then $\Psi_f$ is perverse t-exact in the following sense: 
    for every $\CF\in$ $^p\!D^{\leq 0}_{zc}(X^\times)$ and $\CG\in$ $^p\!D^{\geq 0}_{zc}(X^\times)$ which extend to objects of $\DX$, we have $\Psi_f(\CF)\in$ $^p\!D^{\leq 0}_{zc}(X_0)$ and $\Psi_f(\CG)\in$ $^p\!D^{\geq 0}_{zc}(X_0)$.\\
    (2) For $\CF\in$ $D^{(b)}_{zc}(X^\X)$ which extends to an object of $\DX$, there is a canonical isomorphism $g: \Psi_f\mathbb{D}\CF\isoto(\mathbb{D}\Psi_f\CF)(1)$ in $D(X_0)$ commuting with monodromies.
\end{theorem}

\begin{proof}
    (1) We show $\Psi_f$ is perverse right t-exact by reducing to the unipotent part (Lemma \ref{lem_psiuniistexact}). The left t-exactness follows from the same argument. Let $\CF\in$ $^p\!D^{\leq 0}_{zc}(X^\times,F)$ be such that it extends to an object of $D^{(b)}_{zc}(X,F)$. We want to show $\Psi_f(\CF)\in$ $^p\!D^{\leq 0}_{zc}(X_0,F)$. We may assume $X$ is quasi-compact. Fix a topological generator $t$ of $\pi^{\mathrm{fin}}_1(\Gm,1)$. Then, there exists a finite field extension $F'/F$ such that the $t$ action on $\Psi_f(\CF)$ has a Jordan block decomposition after the extension of scalars $(-)\otimes_F F'$, \textit{i.e.}, there exists a canonical decomposition $\Psi_f(\CF)\otimes_F F'\simeq \oplus_{\lambda\in F'^\X}(\Psi_f(\CF)\otimes_F F')^\lambda$ (note this is a finite direct sum), such that $t-\lambda$ acts nilpotently on $(\Psi_f(\CF)\otimes_F F')^\lambda$ and invertibly on other components. In the following, we abbreviate $(-)\otimes_F F'$ by $(-)'$.\\\\
    \underline{Claim}. Let $\CG\in D^{(b)}_{zc}(-,F)$. Then,\\ 
    (i) there is a canonical isomorphism $(\DD\CG)'\isoto\DD(\CG')$. Here the first $\DD$ is over $F$ and the second is over $F'$.\\
    (ii) $\CG\in$ $^p\!D^{\leq 0}_{zc}(-,F)$ (resp. $^p\!D^{\geq 0}_{zc}(-,F)$) if and only if $\CG'\in$ $^p\!D^{\leq 0}_{zc}(-,F')$ (resp. $^p\!D^{\geq 0}_{zc}(-,F')$).\\
    (iii) There is a canonical isomorphism $\Psi_f(\CG)'\isoto\Psi_f(\CG')$.\\\\ 
    \underline{Proof}. (i) The map is induced from the adjunction $Hom_{F'}((\DD\CG)\otimes_FF',\DD(\CG'))\simeq Hom_{F}(\DD\CG,\DD(\CG'))$ and the natural map $\CG\rightarrow\CG\otimes_F F'$. To check this is an isomorphism, it suffices to check as $F$-sheaves. We have $(\DD\CG)'=\RHom_F(\CG,\omega_F)\otimes_FF'$, and $\DD(\CG')=\RHom_{F'}(\CG\otimes_FF',\omega_{F'})\simeq\RHom_{F}(\CG,\omega_{F'})$ as an $F$-sheaf. The claim then follows from the fact that $\omega_{F'}\simeq\omega_F\otimes_F{F'}$.\footnote{This fact holds in great generality, see \cite[\nopp 6.5.3]{liu_enhanced_2024} and the construction of the dualising complex in our context (\cite[\nopp 3.21]{bhatt_six_2022}).}\\(ii) The statement is clear for $^p\!D^{\leq 0}_{zc}(-,F)$, and for $^p\!D^{\geq 0}_{zc}(-,F)$ this follows from (i).\\
    (iii) The map is constructed similarly as in (i). To check it is an isomorphism, it suffices to check as $F$-sheaves. Since $\underline{F'}$ is isomorphic to a finite direct sum of $\underline{F}$, this is clear. This completes the proof of the Claim.\\\\
    By (ii), it suffices to show $\Psi_f(\CF')\in$ $^p\!D^{\leq 0}_{zc}(X_0,F')$. For $\lambda_0\in F'$, let $\CL^{\lambda_0}$ denote the rank $1$ $F'$-local system on $\Gm$ corresponding to the representation $\pi^{fin}_1(\Gm,1)\rightarrow \mathrm{Aut}_{F'}(F')$, $t\mapsto \lambda_0$. Consider $\Psi_f(\CF'\otimes_{F'} f^*\CL^{\lambda_0})$ (from now on, the tensors will be over $F'$ and we omit the subscript).\\\\
    \underline{Claim}. $\Psi_f(\CF'\otimes f^*\CL^{\lambda_0})\simeq\Psi_f(\CF')^{[\lambda_0]}$, where $(-)^{[\lambda_0]}$ denotes twisting the monodromy action by multiplying by $\lambda_0$.\\\\
    \underline{Proof}. (Notations as in Definition \ref{def_van}.) Let $n\in\ZZ_{\geq1}$ be the smallest such that $\CL$ is trivialised by the $n$-th Kummer covering $e_n$. Then $\Psi_f(\CF'\otimes f^*\CL^{\lambda_0})\simeq i^*\varinjlim_{n|m} j_{m*}j_m^*(\CF'\otimes f^*\CL^{\lambda_0})[-1]\simeq i^*\varinjlim_{n|m} j_{m*}((j_m^*\CF')\otimes (j_{mn}^*\underline{F}'_{X_n^\X}))[-1]$, where $j_{mn}$ is the transition map $X^\X_m\rightarrow X^\X_n$, and the last map is induced by the unique isomorphism $e_n^*\CL\simeq\underline{F}'_{\Gm}$ which restricts to $id$ at $1$. So $\Psi_f(\CF'\otimes f^*\CL^{\lambda_0})\simeq\Psi_f(\CF')$ as sheaves. Note $\underline{F}'_{X_n^\X}$ is equipped with an $\mathrm{Aut}(e_n)$-action corresponding to $\lambda_0$, it follows that the monodromy on the right-hand side is twisted by $\lambda_0$. This completes the proof of the Claim.\\\\ 
    It follows that $\Psi_f(\CF'\otimes f^*\CL^{\lambda_0})^\lambda\simeq\Psi_f(\CF')^{\lambda/\lambda_0}$. Take $\lambda=1$, we get $\Psi^u_f(\CF'\otimes f^*\CL^{\lambda_0}) \simeq \Psi_f(\CF')^{1/\lambda_0}$. By Lemma \ref{lem_psiuniistexact}, this lies in $^p\!D^{\leq 0}_{zc}(X_0,F')$. As this holds for all $\lambda_0$, we get $\Psi_f(\CF')\in$ $^p\!D^{\leq 0}_{zc}(X_0,F')$.\\\\
    (2) The map $g: \Psi_f\mathbb{D}\CF\isoto(\mathbb{D}\Psi_f\CF)(1)$ is defined as in \cite[\nopp 4.3]{illusie_expose_1994}, which we recall: by the tensor-Hom adjunction, it suffices to construct a map $\Psi_f(\DD\CF)\otimes\Psi_f(\CF)\rightarrow\omega_{X_0}(1)$. The lax symmetric monoidal structure on $\psi_f$ and the canonical pairing $(\DD\CF)\otimes\CF\rightarrow\omega_X$ induces $\Psi_f(\DD\CF)\otimes\Psi_f(\CF)[1]\xrightarrow{\mathrm{lsym}}\Psi_f((\DD\CF)\otimes\CF)\xrightarrow{\mathrm{pairing}}\Psi_f\omega_X$. As $\omega_X\simeq f^!\omega_{\AAA^1}$, it remains to construct a map $m: \Psi_ff^!\omega_{\AAA^1}\rightarrow f_0^!\Psi_{id}\omega_{\AAA^1}$. Its composition with $f_0^!\Psi_{id}\omega_{\AAA^1}\simeq f_0^!\omega_{\{0\}}(1)[1]\simeq \omega_{X_0}(1)[1]$ gives the desired map. Here $f_0: X_0\rightarrow \{0\}$ is the restriction of $f$, and the first isomorphism is induced by $\mathrm{sp}: i^*\omega_{\AAA^1}[-1]\isoto\Psi_{id}\omega_{\AAA^1}$. The map $m$ is obtained as the $\varinjlim$ of the following composition map on finite levels (notations as in Definition \ref{def_van}): $i^*j_{n*}j_n^*f^!\simeq i^*j_{n*}f^!e_{n}^*\simeq i^*f_0^!e_{n*}e_{n}^*\rightarrow f_0^!i_0^*e_{0n*}e_{0n}^*$. Here $i_0$ denotes the inclusion $\{0\}\hookrightarrow\AAA^1$.\\\\
    Denote the cone of this map by $\CCC(\CF)$. One can check that the formation of $\CCC(\CF)$ commutes with proper pushforwards (\cite[\nopp 4.3.7]{illusie_expose_1994}) and analytification (use \cite[\nopp 3.16]{bhatt_six_2022}). We show that this is an isomorphism by reducing to the case of simple normal crossings and then apply comparison with the algebraic case, where the result is well-known (see, \textit{e.g.}, \cite[\nopp 3.1]{lu_duality_2019}).\\\\
    We may assume $X$ is quasi-compact. By induction on the amplitude, we reduce to the case $\CF\in\mathrm{Sh}_{zc}(X)$. As the claim clearly holds for sheaves supported on $X_0$, we may assume $\CF_{X_0}=0$. Replace $X$ by $\mathrm{supp}(\CF)$, we may assume $\mathrm{supp}(\CF)=X$. Note that $X_0$ is now a divisor. Then, there exists a dense Zariski open $j: U\subseteq X$ with complement $Z\hookrightarrow X$, such that $Z$ contains $X_0$, $U$ is smooth, and $\CF$ is lisse on $U$. By standard dévissage and induction on $\dim \mathrm{supp}(\CF)$, we may assume $\CF=j_!\CF_U$. Apply resolution of singularities (\cite[\nopp 1.1.13.(i), 1.1.11]{temkin_functorial_2018}) to $Z\hookrightarrow X$, we get a proper surjective map $\pi: X'\rightarrow X$ such that $X'$ is smooth, $\pi$ is an isomorphism over $U$, and $Z':=X'\X_X Z$ is strictly monomial. Denote $j': U':=X'\X_X U\hookrightarrow X'$. Then $\CF\simeq \pi_*(j'_!\CF_{U'})$, and it suffices to prove $\CCC(j'_!\CF_{U'})=0$. In summary, we have reduced to the following situation (separate notations from above):\\\\
    Let $f: X\rightarrow\AAA^1$ be a map of rigid analytic varieties, with $X$ quasi-compact and smooth. Let $Z=\Sigma_{i\in I} H_i \hookrightarrow X$ be a simple normal crossings divisor for some finite set $I$. Assume $X_0:=f^{-1}(0)$ is of the form $\Sigma_{i\in J}n_iH_i$, for a subset of $J$ of $I$ and $n_i\in\ZZ_{\geq 1}$. Denote $j: U=(X-Z)\hookrightarrow X$. Given an $\CF=j_!\CL$, where $\CL$ is a local system in degree $0$ on $U$, we would like to show $\CCC(\CF)=0$. As $\CCC(\CF)=0$ is Zariski-constructible (Proposition \ref{prop_phipreserveszc}), it suffices to show its stalks are $0$ at each classical point $x\in X_0$. Suppose it is contained in $r$ irreducible components of $X_0$ and $r+r'$ irreducible components of $Z$. Then, there exist functions $(T_1,...,T_n)$ on some open neighbourhood $V\subseteq X$ of $x$, such that $V\xrightarrow{(T_1,...,T_n)}\AAA^n$ is \etale, $f_V=uT_1^{n_1}...T_r^{n_r}$ with $u$ a unit on $V$, $Z\cap V=V(T_1...T_{r+r'})$, and $x=V(T_1...T_n)$. Pass to an \etale neighbourhood of $x\in X$ as in the diagram in the proof of Proposition \ref{prop_phipreserveszc} and modify $T_1$, we may assume $f_V=T_1^{n_1}...T_r^{n_r}$. As \etale maps are local isomorphisms at classical points, we may further assume $X=D^n=\Spa(K\langle T_1,...,T_n\rangle)$, then $U=X-V(T_1...T_{r+r'})$, and $Z$ and $f$ are given by the same formulas as above. $\CL$ is trivialised by some \etale covering $Y\rightarrow U$. By Abhyankar’s Lemma in our context (\cite[Proposition 2.1]{li_logarithmic_2019}, take $S$ there to be $H_1\cap...\cap H_{r+r'}$), after base-change by a Kummer covering $X':=\Spa(K\langle z_1,...,z_n\rangle)\rightarrow X$, $\{z_1^N,...,z^N_{r+r'},z_{r+r'+1},...,z_n\}\mapsfrom \{T_1,...,T_n\}$, $Y\rightarrow U$ becomes split over a neighbourhood of $0\in X'$. This implies that, in some polydisc open neighbourhood $W$ of $x$ in $X$, $\CL$ is trivialised by the Kummer covering $X'_W$. On $W$, $\CL$ is the descent of some constant sheaf on $X'_W$, with descent data given by some representation of $\mathrm{Aut}(X'_W/W)$. The same covering and representation determines a local system $\widetilde{\CL}$ on $\AAA^n-V(T_1...T_{r+r'})$, which is isomorphic to $\CL$ on $W$. It suffices to show $\CCC(\widetilde{\CL})=0$. Then the situation is algebraisable, and the claim follows from its well-known algebraic counterpart.
\end{proof}

\begin{remark}
    Here is a possible alternative approach to proving the map $g: (\Psi_f\DD)(1)\rightarrow\DD\Psi_f$ is an isomorphism when $\Lambda=\FF_{\ell^r}$: by perverse t-exactness, it suffices to consider perverse sheaves. We want to show the pairing $\Psi_f(\DD\CF)\otimes\Psi_f(\CF)\rightarrow\omega_{X_0}(1)$ induced by $g$ is perfect. Passing to a finite extension of scalars, we may assume $\Psi_f(\CF)$ and $\Psi_f(\DD\CF)$ completely decompose as $\Psi_f(\CF)\simeq \oplus_{\lambda\in F^\X}\Psi_f(\CF)^{\lambda}$ and $\Psi_f\DD(\CF)\simeq\oplus_{\lambda\in F^\X}(\Psi_f\DD(\CF))^{\lambda}$. Since the pairing is $\pi^{\mathrm{fin}}_1(\Gm,1)$-equivariant, the induced pairing $\Psi_f(\DD\CF)^{\lambda}\otimes\Psi_f(\CF)^{\lambda'}\rightarrow\omega_{X_0}(1)$ is zero unless $\lambda\lambda'=1$. Twist $\CF$ by $\CL^{\lambda_0}$ as in Theorem \ref{thm_psi_t_exact_duality}.1, we reduce to showing $\Psi^u_f(\DD\CF)\otimes\Psi^u_f(\CF)\rightarrow\omega_{X_0}(1)$ is perfect. The proof will be complete as soon as one can show this pairing coincides with (or is homotopic to) the pairing induced by Lemma \ref{lem_beilinsonpsiduality}.
\end{remark}

We now extend Theorem \ref{thm_psi_t_exact_duality}.1 to the case $\Lambda=\ZZ/\ell^r$, $\ell\neq p$. In the following, we denote $(-)\otimes_{\ZZ/\ell^r}\ZZ/\ell$ by $(-)/\ell$. Recall that an $\CF$ in $\Dbzc(X,\ZZ/\ell^r)$ is called \underline{tor-finite} if $\CF/\ell$ is in $\Dbzc(X,\FF_\ell)$. This is equivalent to $\CF_{\bar{x}}/\ell\in D^b(\FF_\ell)$ for each classical point $\bar{x}\rightarrow X$. The full subcategory of tor-finite sheaves are denotes by $\Dbzctf(X,\ZZ/\ell^r)$. Similarly for $^p\!D^{\leq 0}_{zctf}(X,\ZZ/\ell^r)$ and $^p\!D^{\geq 0}_{zctf}(X,\ZZ/\ell^r)$.

\begin{lemma}\label{lem_j_coho_fini}
    Let $X$ be a quasi-compact rigid analytic variety, $i: Z\hookrightarrow X$ be a Zariski-closed immersion, and $j: U\hookrightarrow X$ be its complement. Then $j_*$ has finite cohomological dimension. It follows that for each $\CF\in D^b_{zc}(X,\ZZ/\ell^r)$, we have canonical isomorphisms $(j_*\CF_U)/\ell\isoto j_*(\CF_U/\ell)$ and $(i^!\CF)/\ell\isoto i^!(\CF/\ell)$. In particular, if $\CF\in D^b_{zctf}(U,\ZZ/\ell^r)$ can be extended to an object of $D^b_{zc}(X,\ZZ/\ell^r)$ (resp. $\CG\in D^b_{zctf}(X,\ZZ/\ell^r)$), then $j_*\CF\in D^b_{zctf}(X,\ZZ/\ell^r)$ (resp. $i^!\CF\in D^b_{zctf}(Z,\ZZ/\ell^r)$).
\end{lemma}

\begin{proof}
    Let $\CF$ be a torsion abelian sheaf on $U$. We would like to show $(j_*\CF)_{\bar{x}}\in D(\ZZ)$ is bounded uniformly in $\CF$ and the geometric point $\bar{x}\rightarrow Z$. As $(j_*\CF)_{\bar{x}}\simeq \varinjlim_V R\Gamma(V\X_X U,\CF)$ where $V$ ranges over \etale neighbourhoods of $\bar{x}$ in $X$, it suffices to bound $R\Gamma(V\X_X U,\CF)$ uniformly in $\CF$, $\bar{x}$, and a cofinal system of $V$. As $Z$ is Zariski-closed, there exists an affinoid open neighbourhood $W=\Spa(A)$ of (the image of) $\bar{x}$ in $X$ and $f_1,...,f_r\in A$, for some minimal $r$ (depending on $W$), such that $Z\cap W=V(f_1,...,f_r)$. Let $\{\epsilon_a\}_{a\in\NN}$ be a sequence in $|K^\X|$ converging to $0$. We may assume $V$ is affinoid and lies over $W$. Denote the pullbacks of $f_1,...,f_r$ to $V$ by $f_1',...,f_r'$. Consider the affinoids $V_{\epsilon_a,i}=\{|f'_i|\geq \epsilon_a\}\subseteq V$, and the cover $V_{\epsilon_a}=\cup_iV_{\epsilon_a,i}$. Each $R\Gamma(V_{\epsilon_a},\CF)$ is in $D^{[0,2\dim(X)+r-1]}(\ZZ)$ by the Čech spectral sequence and the bound on cohomological dimension in \cite[\nopp 2.8.3]{huber_etale_1996}. Since $V\X_X U=\cup_{a\rightarrow\infty}V_{\epsilon_a}$, we get $R\Gamma(V\X_X U,\CF)\simeq R\varprojlim_a R\Gamma(V_{\epsilon_a},\CF)$ is in $D^{[0,2\dim(X)+r]}(\ZZ)$. $r$ depends only on $W$, so is globally bounded by the quasi-compactness of $X$.\\\\
    To see the “It follows” part, first note that the statement about $i^!$ will follow from that about $j_*$ and the distinguished triangle $i^!\CF\rightarrow i^*\CF\rightarrow i^*j_*\CF_U\rightarrow$, so it suffices to show the canonical map $(j_*\CF_U)/\ell\rightarrow j_*(\CF_U/\ell)$ is an isomorphism. In the following, denote by $^o\!f^*:\mathrm{Sh}(X,\ZZ/\ell^r)\rightleftarrows\mathrm{Sh}(X,\FF_\ell):$ $^o\!f_*$ the map between the topoi, and by $f^*$, $f_*$ the derived functors (so $(-)/\ell=f^*(-)$). By induction on the amplitude we may assume $\CF\in\mathrm{Sh}_{zc}(X,\ZZ/\ell^r)$. Then $\CF$ can be written as a finite successive extension of objects of the form $f_*\CG$, with $\CG\in\mathrm{Sh}_{zc}(X,\FF_\ell)$. It suffices to prove $f^*j_*f_*\CG\isoto j_*f^*f_*\CG$. Note: (i) $f^*f_*(-)\simeq\oplus_{i\in\NN}(-)[i]$ (use the resolution $(\cdots\rightarrow\ZZ/\ell^r\xrightarrow{\ell^{r-1}}\ZZ/\ell^r\xrightarrow{\ell}\ZZ/\ell^r)\isoto\FF_\ell$), and (ii) $\CH^n(j_*\CG[i])=0$ for $i<n-c$, where $c$ is some constant depending only on $X$ and $Z$ in $X$ (because $j_*$ has finite cohomological dimension). We may then compute: $\CH^n(f^*j_*f_*\CG)\simeq\CH^n(f^*f_*j_*\CG)\simeq\CH^n(\oplus_{i\in\NN}j_*\CG[i])\simeq\CH^n(\oplus_{n-c\leq i\leq n}j_*\CG[i])$, and $\CH^n(j_*f^*f_*\CG)\simeq\CH^n(j_*\oplus_{i\in\NN}\CG[i])\simeq\CH^n(\oplus_{n-c\leq i\leq n}j_*\CG[i])$ (the last step uses (ii) and a standard spectral sequence argument). This finishes the proof.
\end{proof}

\begin{lemma}\label{lem_perv_check_mod_l}
    Let $X$ be a rigid analytic variety, and $\CF\in\Dbzctf(X,\ZZ/\ell^r)$. Then $\CF\in$ $^p\!D^{\leq 0}_{zctf}(X,\ZZ/\ell^r)$ if and only if $\CF/\ell\in$ $^p\!D^{\leq 0}_{zc}(X,\FF_\ell)$. 
\end{lemma}

\begin{proof}
    This follows from the characterisation (2) in Lemma \ref{lem_perversecriterion} and the following two facts: (i) $R\Gamma(Y,\CG)/\ell\isoto R\Gamma(Y,\CG/\ell)$ for any affinoid $Y$ and $\CG\in D^b_{zc}(Y,\ZZ/\ell^r)$ (this is proved in the same way as in the last paragraph of Lemma \ref{lem_j_coho_fini}); (ii) $M\in D^{\leq0}(\ZZ/\ell^r)$ if and only if $M/\ell\in D^{\leq0}(\ZZ/\ell)$ (because the underived tensor product is right exact).
\end{proof}

\begin{lemma}\label{lem_psi_mod_l}
    Let $f: X\rightarrow\AAA^1$ be a map of rigid analytic varieties, and $\CF\in\Dbzc(X,\ZZ/\ell^r)$. Then there is a canonical isomorphism $\psi_f(\CF)/\ell\isoto \psi_f(\CF/\ell)$. In particular, if $\CF\in\Dbzctf(X)$, then $\psi_f(\CF)\in\Dbzctf(X_0)$. Same for $\phi_f$.
\end{lemma}

\begin{proof}
    For $\psi_f$, this follows from its definition and the fact that $(-)/\ell$ commutes with $\varinjlim$, $i^*$, $j_*$ (Lemma \ref{lem_j_coho_fini}), $p_{n*}$, and $p_n^*$. The results for $\phi_f$ then follows from the $can$ triangle.
\end{proof}

\begin{lemma}\label{lem_D_mod_l}
    Let $X$ be a rigid analytic variety, and $\CF\in\Dbzctf(X,\ZZ/\ell^r)$. Then there is a canonical isomorphism $(\DD\CF)/\ell\isoto\DD(\CF/\ell)$. In particular, $\DD\CF\in\Dbzctf(X)$.
\end{lemma}

\begin{proof}
    We omit the standard construction of the map. We may assume $X$ is quasi-compact. Then there exist $j: U\subseteq X$ a smooth dense Zariski open on which $\CF$ is lisse. Denote its complement by $i: Z\hookrightarrow X$. $Z$ has strictly smaller dimension than $X$, so by induction on dimension and standard dévissage (note one needs to use $(j_*(-))/\ell\isoto j_*((-)/\ell)$ from Lemma \ref{lem_j_coho_fini}), it suffices to show $(\DD\CF_U)/\ell\rightarrow\DD(\CF_U/\ell)$ is an isomorphism. Note $\CF_U\in D^b_{lisse, tf}(U)$. By passing to some \etale covering, we may assume $\CF_U\cong p^*M$, where $p$ is $U\rightarrow \Spa(K)$ and $M$ is in $D^b_{ctf}(\ZZ/\ell^r)$ (\cite[the proof of 094G]{Stacks}). Then, $(\DD p^*M)/\ell\simeq (p^!\DD M)/\ell\simeq p^!((\DD M)/\ell)$ (note the smoothness of $U$ is used to write $p^!$ as a twist and shift of $p^*$ to commute with $(-)/\ell$), and $\DD((p^*M)/\ell)\simeq \DD p^*(M/\ell)\simeq p^!\DD(M/\ell)$. The question is now reduced to the case of a point, which is clear.
\end{proof}

\begin{corollary}\label{cor_ptexactphigeneralcoefficient}
    Let $f: X\rightarrow\AAA^1$ be a map of rigid analytic varieties. Assume $\Lambda=\ZZ/\ell^r$, $\ell\neq p$. Then $\Psi_f$ is perverse t-exact, in the following sense: 
    for every $\CF\in$ $^p\!D^{\leq 0}_{zctf}(X^\times)$ and $\CG\in$ $^p\!D^{\geq 0}_{zctf}(X^\times)$ which extend to objects of $\DX$, we have $\Psi_f(\CF)\in$ $^p\!D^{\leq 0}_{zctf}(X_0)$ and $\Psi_f(\CG)\in$ $^p\!D^{\geq 0}_{zctf}(X_0)$.
\end{corollary}

\begin{proof}
     By duality (Theorem \ref{thm_psi_t_exact_duality}.2) and Lemma \ref{lem_D_mod_l}, it suffices to show right t-exactness, \textit{i.e.}, for every $\CF\in$ $^p\!D^{\leq 0}_{zctf}(X^\times)$ which extends to an object of $\DX$, to show $\Psi_f(\CF)\in$ $^p\!D^{\leq 0}_{zctf}(X_0)$. By Lemma \ref{lem_perv_check_mod_l} and Lemma \ref{lem_psi_mod_l}, this is equivalent to $\Psi_f(\CF/\ell)\in$ $^p\!D^{\leq 0}_{zctf}(X_0,\FF_\ell)$, which is true by Lemma \ref{lem_perv_check_mod_l}, and Theorem \ref{thm_psi_t_exact_duality}.1.
\end{proof}

Finally, we discuss vanishing cycles.
\begin{theorem}\label{thm_phi_t_exact_duality}
    Let $f: X\rightarrow\AAA^1$ be a map of rigid analytic varieties.\\
    (1) $\Phi_f: D^{(b)}_{zctf}(X)\rightarrow D^{(b)}_{zctf}(X_0)$ is perverse t-exact;\\
    (2) For $\CF\in$ $D^{(b)}_{zc}(X)$, there is a natural isomorphism $h: \Phi_f\mathbb{D}\CF\isoto(\mathbb{D}\Phi_f\CF)\miwa(1)$ in $D(X_0\bar\X B\mu)$. We have a commutative diagram Diagram \ref{eqn_verdict_duality}. The $\mathrm{can}$ and $\mathrm{var}$ triangles are dual to each other in this sense.
\end{theorem}

Here $\miwa$ denotes the Iwasawa twist (Definition \ref{def_iwasawa}, Remark \ref{rmk_t_non_t_decomposition}), and $\mathbb{D}\Phi_f\CF$ means $\RHom_{X_0}(\Phi_f\CF,\omega_{X_0})$ equipped with the natural $\mu$-action, viewed as an object of $D(X_0\bar\X B\mu)$.

\begin{proof}
    (1) The perverse t-exactness of $\Phi_f$ follows from the result for $\Psi_f$, the \textit{can} and \textit{var} triangles, and the fact that $i^*$ and $i^!$ preserve $\Dbzctf$ (Lemma \ref{lem_j_coho_fini}). Namely, let $\CF\in$ $^p\!D^{\leq 0}_{zctf}(X)$ (resp. $^p\!D^{\geq 0}_{zctf}(X)$). We know $i^*\CF\in$ $^p\!D^{\leq 0}_{zctf}(X_0)$ (resp. $i^!\CF\in$ $^p\!D^{\geq 0}_{zctf}(X_0)$). The perverse cohomology long exact sequence associated to $\Psi_f(\CF)\xrightarrow{\mathrm{can}}\Phi_f(\CF)\rightarrow i^*\CF\rightarrow$ (resp. $i^!\CF\rightarrow\Phi_f(\CF)\xrightarrow{\mathrm{var}}\Psi_f(\CF)\miwa\rightarrow$) then shows $\Phi_f(\CF)\in$ $^p\!D^{\leq 0}_{zctf}(X_0)$ (resp. $^p\!D^{\geq 0}_{zctf}(X_0)$).\\\\
    (2) We will construct a natural transformation $h: \Phi_f\mathbb{D}\rightarrow(\mathbb{D}\Phi_f)\miwa(1)$ as functors $\Dbzc(X)\rightarrow \Dbzc(X_0\bar\X B\mu)$,\footnote{Here $\Dbzc(X_0\bar\X B\mu)$ is the full subcategory of objects whose underlying sheaf is in $\Dbzc(X_0)$.} and show it is an isomorphism. Our method is inspired by \cite[\nopp 5.2.3]{saito_modules_1988}. \\\\
    Recall the \textit{can} and \textit{var} triangles in Corollary \ref{cor_u_nu_decomp}. Apply $\piwa(-1)$ to the \textit{var} triangle, using Lemma \ref{lem_iwasawa_basics}.2, we get $\Phi\piwa(-1)\xrightarrow{}\Psi(-1)\xrightarrow{\mathrm{cosp}\piwa(-1)} i^![1]\rightarrow$.\footnote{Here and in the following, we will call a natural transformation (\textit{e.g.} cosp) and its shifts by the same names (but will record the Tate and Iwasawa twists). We also omit the subscript $f$.} Apply $\DD$, we get $(\DD i^!)[-1]\xrightarrow{}(\DD\Psi)(1)\xrightarrow{}(\DD\Phi)\miwa(1)\rightarrow$. It suffices to construct the following commutative diagram of functors from $\DX$ to $\Dbzc(X_0\bar\X B\mu)$, the cone of which will give us $h: \Phi\DD\rightarrow(\DD\Phi)\miwa(1)$. Here $g$ is the isomorphism in Theorem \ref{thm_psi_t_exact_duality}.2.
\begin{equation}\label{eqn_phi_square_duality}
\begin{tikzcd}
	{(i^*\mathbb{D})[-1]} && {\Psi\mathbb{D}} \\
	{(\mathbb{D}i^!)[-1]} && {(\mathbb{D}\Psi)(1)}
	\arrow["\mathrm{sp}", from=1-1, to=1-3]
	\arrow["\alpha"', dashed, from=1-1, to=2-1]
	\arrow["g", from=1-3, to=2-3]
	\arrow["\simeq"', from=1-3, to=2-3]
	\arrow["{\mathbb{D}(\mathrm{cosp}(1)^\tau(-1))}"', from=2-1, to=2-3]
\end{tikzcd}
\end{equation}
Denote $\DD\CF$ and $\CF$ by $K$ and $L$, respectively. Equivalently, we need to construct a natural pairing
$\alpha': i^*K\otimes i^!L\rightarrow\omega_{X_0}$ and a homotopy between the following two pairings on $i^*K\otimes \Psi L$:\footnote{Of course, all tensor products in this proof are over $\Lambda$ instead of $R(=\Lambda[[\mu]])$.}
\begin{equation}\label{eqn_two_pairings}
\begin{tikzcd}
	{i^*K\otimes \Psi L} & {\Psi K\otimes \Psi L[1]} & {\Psi (K\otimes L)} && {\Psi \omega_X} & {f^!_0\Psi\omega_{\mathbf{A}^1}} & {\omega_{X_0}(1)[1]} \\
	&& {i^*K\otimes i^!L(1)[1]}
	\arrow["{sp\oplus id}", from=1-1, to=1-2]
	\arrow["\gamma"{description}, curve={height=12pt}, dotted, from=1-1, to=1-3]
	\arrow["{\mathrm{id}\otimes (\mathrm{cosp}(1)^\tau)}"', curve={height=12pt}, from=1-1, to=2-3]
	\arrow["\mathrm{lsym}", from=1-2, to=1-3]
	\arrow["pairing", from=1-3, to=1-5]
	\arrow["m", from=1-5, to=1-6]
	\arrow["\simeq", from=1-6, to=1-7]
	\arrow["{\alpha'}"', curve={height=12pt}, from=2-3, to=1-7]
\end{tikzcd}
\end{equation}
Here, in the first line, $f_0^!$, $\mathrm{lsym}$, $\mathrm{pairing}$, $m$ and “$\simeq$” are as in the definition of $g$ (Theorem \ref{thm_psi_t_exact_duality}.2), and we have labelled the composition $\mathrm{lsym}\circ(\mathrm{sp}\oplus \mathrm{id})$ by $\gamma$ for later reference.\\\\
Recall the diagram in Definition \ref{def_sp_can_var}. Denote the following commutative square by $\Upsilon(-)$. 
\[\begin{tikzcd}
	{i^*} & \psi \\
	0 & {\psi(-1)^\tau}
	\arrow["\mathrm{sp}", from=1-1, to=1-2]
	\arrow[from=1-1, to=2-1]
	\arrow["\iota", from=1-2, to=2-2]
	\arrow[from=2-1, to=2-2]
\end{tikzcd}\]
\underline{Step 1}. We construct natural maps $i^*K\otimes \Upsilon L\xrightarrow{(i)} \Upsilon(K\otimes L)\xrightarrow{(ii)} \Upsilon\omega_X\xrightarrow{(iii)} f_0^!\Upsilon\omega_{\AAA^1}$.\\\\
(i) is constructed by the following diagram:
\begin{equation}\label{eqn_THE_diagram}
\begin{tikzcd}
	{i^*K\otimes i^*L} & {i^*K\otimes\psi L} & {i^*K\otimes(\psi L(-1)^\tau)} \\
	& {i^*K\otimes\psi L} & {(i^*K\otimes\psi L)(-1)^\tau} \\
	{i^*(K\otimes L)} & {\psi (K\otimes L)} & {\psi (K\otimes L)(-1)^\tau}
	\arrow["{\mathrm{id}\otimes \mathrm{sp}}", from=1-1, to=1-2]
	\arrow["\simeq"', from=1-1, to=3-1]
	\arrow["{\circled{A}}"{description}, draw=none, from=1-1, to=3-2]
	\arrow["{\mathrm{id}\otimes \iota}", from=1-2, to=1-3]
	\arrow["{=}"', from=1-2, to=2-2]
	\arrow["{\circled{B}}"{description}, draw=none, from=1-2, to=2-3]
	\arrow["\beta", , "\simeq"', from=1-3, to=2-3]
	\arrow["\iota", from=2-2, to=2-3]
	\arrow["\gamma"', from=2-2, to=3-2]
	\arrow["{\circled{C}}"{description}, draw=none, from=2-2, to=3-3]
	\arrow["\gamma\miwa", from=2-3, to=3-3]
	\arrow["\mathrm{sp}", from=3-1, to=3-2]
	\arrow["\iota", from=3-2, to=3-3]
\end{tikzcd}
\end{equation}
Recall $\gamma=\mathrm{lsym}\circ (\mathrm{sp}\otimes \mathrm{id})$, $\mathrm{lsym}$ is ($\varinjlim$ of) $i^*$ applied to the natural transformation $(j_{n*}j_n^*)\otimes (j_{n_*}j_n^*)\rightarrow j_{n*}j_n^*(-\otimes-)$ and $\mathrm{sp}$ is ($\varinjlim$ of) $i^*$ applied to the adjunction $\mathrm{id}\rightarrow j_{n*}j_n^*$ (notations as in Definition \ref{def_van}). So a homotopy in \circled{A} is induced by the monoidal natural transformation structure of $\mathrm{id}\rightarrow j_{n*}j_n^*$ (which is encoded in the six-functor formalism). A homotopy in \circled{C} is provided by the functoriality of $\iota: \mathrm{id}\rightarrow \miwa$.\\\\ 
The commutative square $\circled{B}$ is constructed as the cone of the left square in the following commutative diagram of distinguished triangles:
\[\begin{tikzcd}
	{(i^*K\otimes p^*p_*\psi L)} & {i^*K\otimes\psi L} & {i^*K\otimes(\psi L(-1)^\tau)} & {} \\
	{p^*p_*(i^*K\otimes \psi L)} & {i^*K\otimes\psi L} & {(i^*K\otimes\psi L)(-1)^\tau} & {}
	\arrow["{\mathrm{id}\otimes \mathrm{adj}}", from=1-1, to=1-2]
	\arrow["\simeq"', "p^*(\mathrm{proj})", from=1-1, to=2-1]
	\arrow["{\mathrm{id}\otimes \iota}", from=1-2, to=1-3]
	\arrow["{=}"', from=1-2, to=2-2]
	\arrow["{\circled{B}}"{description}, draw=none, from=1-2, to=2-3]
	\arrow[from=1-3, to=1-4]
	\arrow["\beta", "\simeq"', from=1-3, to=2-3]
	\arrow["\mathrm{adj}", from=2-1, to=2-2]
	\arrow["\iota", from=2-2, to=2-3]
	\arrow[from=2-3, to=2-4]
\end{tikzcd}\]
Here, $p$ is the projection $X_0\bar\X B\mu\rightarrow X_0$, “proj” denotes the projection formula map (we have omitted the $p^*$ in $p^*i^*$ to simplify notations), which is an isomorphism by \cite[Corollary 1.20.(a)]{lu_duality_2019}, and the homotopy of the left square is induced by the equality of maps after taking adjoints:\footnote{On the abelian level of $R$-modules, $\beta$ can be described as follows: let $M$ (resp. $I$) be a discrete $\Lambda$-module with a continuous (resp. trivial) $\mu$-action. Identify $I\otimes(M\miwa)$ with $I\otimes Z^1_{cont}(\mu,M)$, and $(I\otimes M)\miwa$ with $Z^1_{cont}(\mu,I\otimes M)$ (Lemma \ref{lem_iwasawa_basics}.1) (the tensor products are underived and over $\Lambda$). Then $\beta$ is $I\otimes(M\miwa)\rightarrow (I\otimes M)\miwa$, $x\otimes (h\mapsto f(h))\mapsto (h\mapsto x\otimes f(h))$. Using the fact that $I$ has the trivial $\mu$-action, one checks that this is a well-defined map of $R$-modules. That it is an isomorphism can be seen by fixing any topological generator $t$ of $\mu$ and using the induced isomorphisms (see the paragraph after Definition \ref{def_iwasawa}) from $I\otimes(M(-1)^\tau)$ and $(I\otimes M)(-1)^\tau$ to $I\otimes M$.
}
\[\begin{tikzcd}
	{i^*K\otimes p_*\psi L} & {p_*(i^*K\otimes\psi L)} \\
	{p_*(i^*K\otimes \psi L)} & {p_*(i^*K\otimes\psi L)}
	\arrow["\mathrm{proj}", from=1-1, to=1-2]
	\arrow["\mathrm{proj}"', from=1-1, to=2-1]
	\arrow["{=}"', from=1-2, to=2-2]
	\arrow["{=}", from=2-1, to=2-2]
\end{tikzcd}\]
This completes the construction of $(i): i^*K\otimes \Upsilon L\rightarrow \Upsilon(K\otimes L)$ (the remaining $2$-morphism is obtained by composition and the $3$-morphism by “filling the cylinder” using the axiom for $\infty$-categories).\\\\
(ii) is given by the functoriality of $\Upsilon(-)$ and the pairing $K\otimes L\rightarrow\omega_X$.\\\\
(iii) is constructed by the following diagram:
\[\begin{tikzcd}
	{i^*f^!\omega_{\mathbf{A}^1}} & {\psi f^!\omega_{\mathbf{A}^1}} && {\psi_0\omega_{\mathbf{A}^1}} \\
	{f_0^!i_0^*\omega_{\mathbf{A}^1}} & {f_0^!\psi_0\omega_{\mathbf{A}^1}} && {\psi_0\omega_{\mathbf{A}^1}}
	\arrow["\mathrm{sp}", from=1-1, to=1-2]
	\arrow[from=1-1, to=2-1]
	\arrow["\iota", from=1-2, to=1-4]
	\arrow["m", from=1-2, to=2-2]
	\arrow["{m(-1)^\tau}", from=1-4, to=2-4]
	\arrow["{f_0^!(\mathrm{sp})}", from=2-1, to=2-2]
	\arrow["{f_0^!(\iota)\simeq\iota}", from=2-2, to=2-4]
\end{tikzcd}\]
Here, $i_0$ denotes the inclusion $\{0\}\hookrightarrow \AAA^1$ and $\psi_0$ denotes the nearby cycle with respect to $\mathrm{id}: \AAA^1\rightarrow\AAA^1$. The homotopy in the right square is given by the functoriality of $\iota$, and that in the left one is induced by the following digram (the first line is as in the construction of $m$ in Theorem \ref{thm_psi_t_exact_duality}.2), where the homotopies are the evident ones:
\[\begin{tikzcd}
	{i^*j_{n*}j_n^*f^!\simeq i^*f_0^!e_{n*}e_{n}^*} & {f_0^!i_0^*e_{0n*}e_{0n}^*} \\
	{i^*f^!=i^*f^!} & {f_0^!i_0^*}
	\arrow[from=1-1, to=1-2]
	\arrow["{i^*f_0^!(\mathrm{adj})}"', shift right=7, from=2-1, to=1-1]
	\arrow["{i^*(\mathrm{adj})f^!}", shift left=7, from=2-1, to=1-1]
	\arrow[from=2-1, to=2-2]
	\arrow["{f_0^!i_0^*(\mathrm{adj})}"', from=2-2, to=1-2]
\end{tikzcd}\]
This completes the construction of $i^*K\otimes \Upsilon L\xrightarrow{(i)} \Upsilon(K\otimes L)\xrightarrow{(ii)} \Upsilon\omega_X\xrightarrow{(iii)} f_0^!\Upsilon\omega_{\AAA^1}$.\\\\
\underline{Step 2}. By taking cones, $(i)$, $(ii)$ and $(iii)$ induce maps among the full diagrams (Definition \ref{def_sp_can_var}). In particular, $(ii)\circ(i)$ induces a map from:\\
\begin{equation}\label{eqn_big_diagram_K_L}
\begin{tikzcd}
	0 & {i^*K\otimes(\Psi L(-1)^\tau)[-1]} & {i^*K\otimes(\Psi L(-1)^\tau)[-1]} \\
	{i^*K\otimes i^*L[-1]} & {i^*K\otimes i^*j_*j^*L[-1]} & { i^*K\otimes i^!L} \\
	{i^*K\otimes i^*L[-1]} & {i^*K\otimes\Psi L} & {i^*K\otimes\Phi L} \\
	0 & {i^*K\otimes(\Psi L(-1)^\tau)} & {i^*K\otimes(\Psi L(-1)^\tau)}
	\arrow[from=1-1, to=1-2]
	\arrow[from=1-1, to=2-1]
	\arrow["{=}", from=1-2, to=1-3]
	\arrow[from=1-2, to=2-2]
	\arrow["{\mathrm{id}\otimes \mathrm{cosp}}", from=1-3, to=2-3]
	\arrow[from=2-1, to=2-2]
	\arrow["{=}", from=2-1, to=3-1]
	\arrow[from=2-2, to=2-3]
	\arrow[from=2-2, to=3-2]
	\arrow[from=2-3, to=3-3]
	\arrow["{\mathrm{id}\otimes \mathrm{sp}}", from=3-1, to=3-2]
	\arrow[from=3-1, to=4-1]
	\arrow["{\mathrm{id}\otimes \mathrm{can}}", from=3-2, to=3-3]
	\arrow["{\mathrm{id}\otimes\iota}", from=3-2, to=4-2]
	\arrow["{\mathrm{id}\otimes \mathrm{var}}", from=3-3, to=4-3]
	\arrow[from=4-1, to=4-2]
	\arrow["{=}", from=4-2, to=4-3]
\end{tikzcd}
\end{equation}
to
\begin{equation}\label{eqn_big_diag_omega}
\begin{tikzcd}
	0 & {\Psi\omega_X(-1)^\tau[-1]} & {\Psi\omega_X(-1)^\tau[-1]} \\
	{i^*\omega_X[-1]} & {i^*j_*j^*\omega_X[-1]} & {i^!\omega_X} \\
	{i^*\omega_X[-1]} & {\Psi\omega_X} & {\Phi\omega_X} \\
	0 & {\Psi\omega_X(-1)^\tau} & {\Psi\omega_X(-1)^\tau}
	\arrow[from=1-1, to=1-2]
	\arrow[from=1-1, to=2-1]
	\arrow["{=}", from=1-2, to=1-3]
	\arrow[from=1-2, to=2-2]
	\arrow["\mathrm{cosp}", from=1-3, to=2-3]
	\arrow[from=2-1, to=2-2]
	\arrow["{=}", from=2-1, to=3-1]
	\arrow[from=2-2, to=2-3]
	\arrow[from=2-2, to=3-2]
	\arrow[from=2-3, to=3-3]
	\arrow["\mathrm{sp}", from=3-1, to=3-2]
	\arrow[from=3-1, to=4-1]
	\arrow["\mathrm{can}", from=3-2, to=3-3]
	\arrow["\iota", from=3-2, to=4-2]
	\arrow["\mathrm{var}", from=3-3, to=4-3]
	\arrow[from=4-1, to=4-2]
	\arrow["{=}", from=4-2, to=4-3]
\end{tikzcd}
\end{equation}
The map from Diagram \ref{eqn_big_diagram_K_L} to Diagram \ref{eqn_big_diag_omega} in particular gives a commutative square (from the top right vertical arrow):
\[\begin{tikzcd}
	{i^*K\otimes(\Psi L(-1)^\tau)[-1]} & {\Psi\omega_X(-1)^\tau[-1]} \\
	{ i^*K\otimes i^!L} & {i^!\omega_X}
	\arrow["{\gamma'}", from=1-1, to=1-2]
	\arrow["{\mathrm{id}\otimes \mathrm{cosp}}"', from=1-1, to=2-1]
	\arrow["\mathrm{cosp}", from=1-2, to=2-2]
	\arrow[from=2-1, to=2-2]
\end{tikzcd}\]
Concatenate with the one induced by $(iii)$, use the isomorphism $\beta$, and apply $\piwa[1]$, we get:
\begin{equation}\label{eqn_pairing_square}
\begin{tikzcd}
	{(i^*K\otimes\Psi L)} & {\Psi\omega_X} & {f_0^!\Psi_0\omega_{\mathbf{A}^1}} \\
	{ i^*K\otimes i^!L(1)[1]} & {i^!\omega_X(1)[1]} & {f_0^!i_0^!\omega_{\mathbf{A}^1}(1)[1]}
	\arrow["\delta", from=1-1, to=1-2]
	\arrow["\epsilon", from=1-1, to=2-1]
	\arrow["m", from=1-2, to=1-3]
	\arrow["{\mathrm{cosp}(1)^\tau}", from=1-2, to=2-2]
	\arrow["{f_0^!(\mathrm{cosp}(1)^\tau)}", from=1-3, to=2-3]
	\arrow[from=2-1, to=2-2]
	\arrow["{\alpha''}"{description}, curve={height=12pt}, dotted, from=2-1, to=2-3]
	\arrow["\sim", from=2-2, to=2-3]
\end{tikzcd}
\end{equation}
As one may verify, the top right arrow is indeed $m$, and the lower right arrow is the usual one ($i^!f^!\simeq f_0^!i_0^!$). The composition of the lower two arrows is labelled as $\alpha''$. We claim that this gives the desired Diagram \ref{eqn_two_pairings}. More precisely, we need to provide natural homotopies for the following three pairs: (a) $\delta$ and $\mathrm{pairing}\circ\gamma$, (b) $\epsilon$ and $\mathrm{id}\otimes (\mathrm{cosp}\piwa)$, (c) $f_0^!(\mathrm{cosp}\piwa)$ and the isomorphism in Diagram \ref{eqn_two_pairings}. Then, Diagram \ref{eqn_two_pairings} is obtained by defining $\alpha'$ (resp. the homotopy) there as $\alpha''$ (resp. the one induced from Diagram \ref{eqn_pairing_square}). In the following computations, we will label terms in Diagram \ref{eqn_big_diagram_K_L} by their rows and columns. For example, the top right $i^*K\otimes(\Psi L\miwa)[-1]$ is term (1,3).\\\\
(a). We compute $\delta$ (which lives at (1,3)) as follows:
\begin{align*}
\delta&=\piwa[1]((i^*K\otimes \Psi L)\miwa[-1]\xrightarrow[\sim]{\beta^{-1}}i^*K\otimes(\Psi L\miwa)[-1]\xrightarrow{\gamma'}\Psi\omega_X(-1)^\tau[-1]) \\
& \simeq \piwa((i^*K\otimes \Psi L)\miwa\xrightarrow[\sim]{\beta^{-1}}i^*K\otimes(\Psi L\miwa)\xrightarrow{\gamma'}\Psi\omega_X(-1)^\tau) \\
& \simeq \piwa((i^*K\otimes \Psi L)\miwa\xrightarrow{(\mathrm{pairing}\circ\gamma)\miwa}\Psi\omega_X(-1)^\tau)\\
& \simeq (i^*K\otimes \Psi L)\xrightarrow{(\mathrm{pairing}\circ\gamma)}\Psi\omega_X
\end{align*}
From the first to the second line, we moved from (1,3) to (4,2) using the equality (1,2)$\rightarrow$(1,3) and the shift, and in the third we used the last column of Diagram \ref{eqn_THE_diagram}.\\\\
(b). Before computing $\epsilon$, we make a preliminary observation. Namely, we have the following canonical isomorphism $(i^*K\otimes \Upsilon L)\piwa\isoto i^*K\otimes ((\Upsilon L)\piwa)$:
\[\begin{tikzcd}
	{(i^*K\otimes i^*L)(1)^\tau[-1]} && {(i^*K\otimes \Psi L)(1)^\tau} && {(i^*K\otimes (\Psi L (-1)^\tau))(1)^\tau} \\
	{i^*K\otimes (i^*L(1)^\tau)[-1]} && {i^*K\otimes (\Psi L (1)^\tau)} && {i^*K\otimes \Psi L}
	\arrow["{(\mathrm{id}\otimes sp)(1)^\tau}", from=1-1, to=1-3]
	\arrow["\simeq"', from=1-1, to=2-1]
	\arrow["{(\mathrm{id}\otimes \iota)(1)^\tau}", from=1-3, to=1-5]
	\arrow["{\beta(1)^\tau}", from=1-3, to=2-3]
	\arrow["\simeq"', from=1-3, to=2-3]
	\arrow["{\beta(1)^\tau}", from=1-5, to=2-5]
	\arrow["\simeq"', from=1-5, to=2-5]
	\arrow["{\mathrm{id}\otimes (\mathrm{sp}(1)^\tau)}", from=2-1, to=2-3]
	\arrow["{\mathrm{id}\otimes (\iota(1)^\tau)}", from=2-3, to=2-5]
\end{tikzcd}\]
where the homotopy on the left is clear (and we will identify the left two terms with $i^*K\otimes i^! L(1)$), and on the right is induced from the functoriality of $\beta$. We leave the details to the reader. It induces an isomorphism between the full diagrams. In particular, we have a commuative square (from the top right vertical arrow):
\begin{equation}\label{eqn_cosp_commutativity}
\begin{tikzcd}
	{(i^*K\otimes (\Psi L (-1)^\tau))(1)^\tau[-1]} & {i^*K\otimes \Psi L[-1]} \\
	{i^*K\otimes i^!L(1)} & {i^*K\otimes i^!L(1)}
	\arrow["{\beta(1)^\tau}", from=1-1, to=1-2]
	\arrow["{(\mathrm{id}\otimes \mathrm{cosp})(1)^\tau}"', from=1-1, to=2-1]
	\arrow["{\mathrm{id}\otimes (\mathrm{cosp}(1)^\tau)}", from=1-2, to=2-2]
	\arrow["\sim", from=2-1, to=2-2]
\end{tikzcd}
\end{equation}
Note the top arrow is indeed $\beta\piwa$ as can be seen by tracing the diagrams from (4,2) to (1,2) then to (1,3). We can now compute $\epsilon$:
\begin{align*}
\epsilon&=\piwa[1]((i^*K\otimes \Psi L)\miwa[-1]\xrightarrow[\sim]{\beta^{-1}}i^*K\otimes(\Psi L\miwa)[-1]\xrightarrow{\mathrm{id}\otimes \mathrm{cosp}}i^*K\otimes i^!L) \\
& \simeq ((i^*K\otimes \Psi L)\xrightarrow[\sim]{\beta^{-1}\piwa}i^*K\otimes(\Psi L\miwa)\piwa\xrightarrow{(\mathrm{id}\otimes \mathrm{cosp})\piwa}i^*K\otimes i^!L(1)[1]) \\
& \simeq (i^*K\otimes \Psi L\xrightarrow{\mathrm{id}\otimes (\mathrm{cosp}\piwa)}i^*K\otimes i^!L(1)[1]) \,\, (\text{using Diagram \ref{eqn_cosp_commutativity}}).
\end{align*}
(c). The isomorphism in Diagram \ref{eqn_two_pairings} is $f_0^!$ applied to $\Psi_0\omega_{\AAA^1}\xrightarrow{\mathrm{sp}^{-1}}i_0^*\omega_{\AAA^1}[-1]\xrightarrow{\mathrm{purity}}i_0^!\omega_{\AAA^1}(1)[1]$. View $\mathrm{purity}\circ \mathrm{sp}^{-1}$ and $\mathrm{cosp}\piwa$ as elements in $\mathrm{Aut}_{D(\Lambda)}(\omega_{\{0\}}(1)[1])\simeq \mathrm{Aut}_{\mathrm{Mod}(\Lambda)}(\Lambda)\simeq \Lambda^\X$ (here $\mathrm{Mod}(\Lambda)$ is the abelian category of $\Lambda$-modules). An explicit computation (use Lemma \ref{lem_iwasawa_basics}.2) shows that both correspond to element $1\in \Lambda^\X$.\footnote{In other words, the composition $i_0^*\underline{\Lambda}_{\AAA^1}\xrightarrow{\mathrm{sp}}\psi_0\underline{\Lambda}_{\AAA^1}\xrightarrow{\mathrm{cosp}\piwa}i_0^!\underline{\Lambda}_{\AAA^1}(1)[2]$ equals $i_0^*\underline{\Lambda}_{\AAA^1}\xrightarrow{\mathrm{purity}}i_0^!\underline{\Lambda}_{\AAA^1}(1)[2]$ in $\mathrm{Mod}(\Lambda)$.} This completes the construction of Diagram \ref{eqn_two_pairings}, hence Diagram \ref{eqn_phi_square_duality}.\\\\
\underline{Step 3}. Taking the cone in Diagram \ref{eqn_phi_square_duality}, we get:
\begin{equation}\label{eqn_verdict_duality}
\begin{tikzcd}
	{(i^*\mathbb{D}\mathcal{F})[-1]} && {\Psi\mathbb{D}\mathcal{F}} && {\Phi\mathbb{D}\mathcal{F}} \\
	{(\mathbb{D}i^!\mathcal{F})[-1]} && {(\mathbb{D}\Psi\mathcal{F})(1)} && {(\mathbb{D}\Phi\mathcal{F})(-1)^\tau(1)}
	\arrow["\mathrm{sp}", from=1-1, to=1-3]
	\arrow["\alpha", from=1-1, to=2-1]
	\arrow["\mathrm{can}", from=1-3, to=1-5]
	\arrow["g", from=1-3, to=2-3]
	\arrow["\simeq"', from=1-3, to=2-3]
	\arrow["h", from=1-5, to=2-5]
	\arrow["{\mathbb{D}(\mathrm{cosp}(1)^\tau(-1))}", from=2-1, to=2-3]
	\arrow["{\mathbb{D}(\mathrm{var}(1)^\tau(-1))}", from=2-3, to=2-5]
\end{tikzcd}
\end{equation}
We now prove that $\alpha$ (hence $h$) is an isomorphism. If $\CF\simeq j_*j^*\CF$, this is clear as both terms are $0$. If $\CF\simeq i_*i^*\CF$, we claim that $\alpha$ can be identified with the usual isomorphism. Indeed, recall as our map $i^*K\otimes \Upsilon L\rightarrow \Upsilon\omega_X$ factors through $i^*K\otimes \Upsilon L\rightarrow \Upsilon(K\otimes L)\rightarrow \Upsilon\omega_X$ where the second arrow is induced by the pairing $K\otimes L\rightarrow\omega_X$. The map between the (2,3) terms in Diagrams \ref{eqn_big_diagram_K_L} and \ref{eqn_big_diag_omega} then factors through $i^*K\otimes i^!L\rightarrow i^!(K\otimes L)\rightarrow i^!\omega_X$. The first map is the usual isomorphism (note the $\Psi$'s in Diagram \ref{eqn_big_diagram_K_L} are $0$), and the second can be identified with the usual pairing $i^*K\otimes i^*L\rightarrow \omega_{X_0}$.
\end{proof}

\begin{remark}[for Theorem \ref{thm_phi_t_exact_duality}.2]
    (1) The same proof works for algebraic varieties over an algebraically closed field, except that in positive characteristic, one should replace $\Psi$ (resp. $\Phi$) above by $\Psi^P$ (resp. $\Phi^P$) (see Remark \ref{rmk_t_non_t_decomposition}, Corollary \ref{cor_u_nu_decomp}) and $\mu$ by $\ZZ_\ell(1)$, then use $^P\!\Psi\simeq\!^P\!\Phi$ and Theorem \ref{thm_psi_t_exact_duality}.2.\\
    (2) We do not know if $\alpha$ always coincides with the usual isomorphism. We also do not know if the usual isomorphism gives a commutative square in Diagram \ref{eqn_phi_square_duality}.\\ 
    (3) It seems that the above proof essentially gives a duality theorem for $\Upsilon$ and its “cone”. It would be interesting to know the precise relation of this with the duality theorems for $\Psi^s$ and $LCo$ in \cite[Theorem 0.1]{lu_duality_2019}.
\end{remark}

\printbibliography

\Addresses

\end{document}